\documentclass[12pt]{amsart}
\usepackage{amsmath,amssymb,amsbsy,amsfonts,amsthm,latexsym,
                        amsopn,amstext,amsxtra,euscript,amscd,mathrsfs,color,bm}
                        
\usepackage{float} 
\usepackage[english]{babel}
\usepackage{mathtools}
\usepackage{todonotes}
\usepackage{url}
\usepackage[colorlinks,linkcolor=blue,anchorcolor=blue,citecolor=blue,backref=page]{hyperref}
\usepackage{enumitem}

\def\cm#1{\textcolor{violet}{#1}}

\newtheorem{theorem}{Theorem}
\newtheorem{lemma}{Lemma}

\newtheorem{prop}{Proposition}

\def\cE{{\mathcal E}}

\def\cH{{\mathcal H}}

\def\cL{{\mathcal L}}
\def\cM{{\mathcal M}}

\def\cR{{\mathcal R}}

\def\C{{\mathbb C}}

\def\Q{{\mathbb Q}}
\def\R{{\mathbb R}}

\def\Z{{\mathbb Z}}

\newtheorem{corollary}{Corollary}

\def\\{\cr}
\def\({\left(}
\def\){\right)}
\def\[{\left[}
\def\]{\right]}
\def\<{\langle}
\def\>{\rangle}

\def\le{\leqslant}
\def\ge{\geqslant}

\def\Hb1{\overline{\cH}_{m}}
\def\Ht1{\widetilde{\cH}_{m}}

\title[Weyl sums with multiplicative coefficients]{Weyl sums with multiplicative coefficients and joint equidistribution}
\author{Matteo Bordignon  }
\address{KTH Royal Institute of Technology, Stockholm \newline and \newline Charles University, Faculty of Mathematics and Physics, Department of Algebra, Sokolovská 83, 186 00 Praha 8, Czech Republic Department of Mathematics}
\email{matteobordignon91@gmail.com}

\author{Cynthia Bortolotto  }
\address{Department of Mathematics, ETH Z\"{u}rich, Switzerland}
\email{cynthia.bortolotto@math.ethz.ch}

\author{Bryce Kerr  }
\address{UNSW Canberra, School of Science}
\email{bryce.kerr@unsw.edu.au}

\begin{document}
 
\begin{abstract} 
In this paper we generalize a result of Montgomery and Vaughan regarding exponential sums with multiplicative coefficients to the setting of Weyl sums.  As applications, we establish a joint equidistribution result for roots of polynomial congruences and polynomial values and obtain some new results for mixed character sums. 
 \end{abstract}
\maketitle

\tableofcontents
\section{Introduction}
Let $A\ge 1$ and $f$ a multiplicative function satisfying $|f(p)| \le A$ for any prime $p$ and $\sum_{n\leq N} |f(n)|^2 \le A^2 N$ for all natural numbers $N$. For $\alpha \in \mathbb{R}$ set
\begin{equation*}
S(\alpha):= \sum_{1\le n \le N} f(n)e(\alpha n),
\end{equation*}
where  $e(x)= \exp(2\pi i x)$.

These sums first appear to be considered by Daboussi~\cite{Dab}, who showed that if $|\alpha-a/q| \le 1/q^2$, $(a,q)=1$ and $ 3 \le q \le (N/\log N)^{\frac{1}{2}}$, then
\begin{equation*}
S(\alpha)\ll \frac{N}{(\log \log q)^{\frac{1}{2}}},
\end{equation*}
and implied constant depending only on $A$.

This result was improved  by Montgomery and Vaughan~\cite[Corollary~1]{M-V} who show that assuming $|\alpha-a/q| \le 1/q^2$, $(a,q)=1$ and $2 \le R \le q \le N/R$, we have 
\begin{equation}
\label{eq:Mvbound1}
S(\alpha)\ll \frac{N}{\log N}+ \frac{N(\log R)^{\frac{3}{2}}}{R^{\frac{1}{2}}}.
\end{equation}
We refer the reader to~\cite[Section~7]{M-V} for a demonstration that the term $N/\log{N}$ is sharp. 

The optimal dependence on $R$ in~\eqref{eq:Mvbound1} is an open problem and has been the subject of a number of works, see for example~\cite{Bach,Bre}, and it is expected the estimate~\eqref{eq:Mvbound1} may be improved to
\begin{equation}
\label{eq:Mvbound2}
S(\alpha)\ll \frac{N}{\log N}+ \frac{N}{R^{\frac{1}{2}}}.
\end{equation}

Recently, Bret\'{e}che and Granville~\cite{BreGra} have studied in detail the sums $S(\alpha)$ on minor arcs. Their estimates suggest the following conjecture (see~\cite[Equation~1.4]{BreGra})
\begin{equation*}
S(\alpha)\ll \frac{N}{\log N}+ \frac{N}{q^{\frac{1}{2}}(1+|\beta|x)}, \quad \alpha=\frac{a}{q}+\beta, \quad (a,q)=1.
\end{equation*}
We also note that~\cite{BreGra} contains some nice applications to circle-method type problems.

The estimate~\eqref{eq:Mvbound1} has important applications to Dirichlet L-functions. For example, Montgomery and Vaughan~\cite{M-V} have shown how it may be combined with the Generalized Riemann Hypothesis (GRH) to obtain a sharp upper bound for Dirichlet L-functions at the point $s=1$. One may also combine~\eqref{eq:Mvbound1} with progress around the Burgess bound to obtain unconditional variants of Montgomery and Vaughan's result and we refer the reader to~\cite{GS,Hildebrand,Hild1} for progress in this direction.


Since the work of Montgomery and Vaughan, exponential sums with multiplicative coefficients have appeared in a number of different contexts and a variety of techniques have been developed to facilitate the reduction to exponential sums over bilinear forms. Some examples include Karatsuba's work on short Kloosterman sums~\cite{Kar}, which has been refined by Korolev, see for example~\cite{Kor}. Bourgain, Sarnak and Ziegler~\cite{BSZ} have established a finite version of Vinogradov's bilinear sum inequality.  Gong and Jia~\cite{GJ} have considered shifted character sums with multiplicative coefficients and Korolev and Shparlinski~\cite{KS} dealt with sums over trace functions with multiplicative coefficients.

In this paper, we revisit the approach of Montgomery and Vaughan and generalise into the setting of sums of the form 
\begin{align}
\label{eq:mv1}
\sum_{n\le N}f(n)e(g(n)),
\end{align}
where $g$ is a polynomial with real coefficients and $f$ is a multiplicative function satisfying 
\begin{align*}
f(p)=O(1), \quad \sum_{n\le N}|f(n)|=O(N), \quad \sum_{n\le N}|f(n)|^2=O(N(\log{N})^{A}),
\end{align*}
for any $A \geq 0$.

 Problems of this sort have previously been considered by  Jiang, L\"{u} and Wang \cite{Y}, who showed that one may replace an assumption on the $\ell_2$ norm
$$\sum_{n\le N}|f(n)|^2\ll N,$$
with an assumption of the form 
$$\sum_{\substack{p\le N \\ p \ \text{prime}}}|f(p)||f(p+h)|\ll \frac{h}{\phi(h)}\frac{N}{(\log{N})^2}.$$
Such a relaxation is significant in the context of $\text{GL}_m$ $L$-functions in the absence of progress towards the Ramanujan conjectures.

Matthiesen~\cite{Math} has considered sums of the form~\eqref{eq:mv1} over polynomial nilsequences with slightly weaker conditions on $f$ than Montgomery and Vaughan. These results were later applied to linear correlations of multiplicative functions~\cite{Math1}. We also mention the recent work~\cite{MMMM} which considers exponential sums with multiplicative functions over nilsequences on average over short intervals.

\section{Main results}

Our first result is as follows.
\begin{theorem}
\label{theo:main}
Let $A>0$ and $c>0$ be real numbers and $f$ be a multiplicative function satisfying 
\begin{align}
\label{eq:fprime}
|f(p)| \le C \quad \text{for each prime $p$},
\end{align}
\begin{align}
\label{eq:fell1}
\sum_{n\le N}|f(n)|=O(N),
\end{align}
\begin{align}
\label{eq:fell2}
\sum_{n\le N}|f(n)|^2=O(N(\log{N})^{A}).
\end{align}
Let $F$  be a polynomial of degree $d\ge 1$ with real coefficients given by 
$$F(x)=\alpha_d x^d+\dots+\alpha_1 x.$$
Let $R\ge 1$ and suppose there exist integers $l, a, q$ with $1\le q\le R$ and $1\le \ell \le d$, $(a,q)=1$ and 
\begin{align*}
\left|\alpha_{\ell}-\frac{a}{q}\right|\le \frac{1}{Rq},
\end{align*} 
Denote $C = \frac{A}{2r}$. Then for any $r >d(d+1)$, we have

\begin{align*}
\sum_{1\le n \le N}f(n)e(F(n))&\ll N\left(\frac{1}{(\log{N})^{1-C }}+(\log{N})^{C}\left(\frac{q}{N^{\ell}}+\frac{1}{q}\right)^{1/4r^2}\right) \\ 
& \quad \quad \quad \quad \quad \quad \quad +(NR^{1/\ell})^{1/2},
\end{align*}
where the implied constant depends on $A$ and $r$.
In particular, if we suppose that 
\begin{equation*}
 (\log{N})^{4r^2  } \le q \le \frac{N^{\ell}}{ (\log{N})^{4r^2 } },
 \end{equation*}
 then
\begin{align}
\sum_{1\le n \le N}f(n)e(F(n))\ll \frac{N}{(\log{N})^{1-C}}.  
\end{align}

\end{theorem}
To demonstrate the precision of the above estimate, in Section \ref{sec:sharp} we prove that for any polynomial $F$ and for all $N$,  there exists $f=f_{F,  N}$ such that
\begin{align*}
\left| \sum_{1\le n \le N}f(n)e(F(n)) \right| \ge \frac{1}{10}\frac{N}{\log{N}}.
\end{align*}

The proof of Theorem~\ref{theo:main} follows the outline of Montgomery and \newline Vaughan~\cite{M-V} and starts with a combinatorial decomposition of multiplicative functions based on M\"{o}bius inversion and reduces the problem to estimating bilinear forms over polynomials with summation restricted to points under the hyperbola. We then use Montgomery and Vaughan's  partition of the parabola into disjoint rectangles to which techniques related to the Vinogradov Mean Value theorem may be applied. It will be fundamental to develop a version of Vinogradov Mean Value theorem for primes in large translated intervals, this is Lemma \ref{lem:vmvtP1}.  We should note that we introduced the general condition \eqref{eq:fell2} with the aiming of using Theorem \ref{theo:main} in the proof of Theorem \ref{theorem2}.
\section{Applications}
\label{sec:applications}
\subsection{Joint equidistribution}
As an application of Theorem \ref{theo:main},  we prove a joint distribution result. Throughout this section, we let $p \in \Z[x]$ be irreducible over $\Q$ of degree $e \ge 2$ and we consider the ratios $v/n$, where $v$ are the roots of the polynomial $p$ modulo $n$
\begin{align*}
p(v) \equiv 0 \bmod{n}.
\end{align*}
Consider the sequence $(g_k)_{k\ge 1}$ of these ratios so that the corresponding denominators are in ascending order.
Hooley proved, in \cite[Theorem 2]{Hooley},  that this sequence is equidistributed in $\R /\Z$.  

We now let $F(x) = \alpha_1 x + \ldots +\alpha_d x^d \in \R[x]$ with $d \ge 1$ and with an irrational coefficient and define
\begin{align*}
A(F,p)_k = ( g_k, F(k'))_{k\ge 1},
\end{align*}
where $g_k$ is as above and for  $g_k = v/n$, $p(v) \equiv 0 \bmod{n}$, we take $k'=n$. We prove the following

\begin{theorem}
\label{theorem2}
The sequence $(A(F,p)_k)_{k\ge 1}$ is equidistributed in $(\R/\Z)^2$.
\end{theorem}
This indicates that the sequence $(F(n))_{n\in \ge 1}$ is somehow not correlated with the sequence $ (g_n)_{n \ge 1}.$ 
\subsection{Mixed character sums}
We next explain how Theorem~\ref{theo:main} is related to sums considered by Enflo~\cite{Enflo}, Chang~\cite{Chang} and Heath-Brown and Pierce~\cite{H-B}.
Theorem~\ref{theo:main} implies large exponential sums must correspond to pretentious multiplicative functions.
\begin{corollary}
\label{cor:large}
With notation and conditions as in Theorem~\ref{theo:main}, let $r\ge d(d+1)$ and suppose that 
\begin{align*}
\left| \sum_{1\le n \le N}f(n)e(F(n)) \right| \gg \frac{N}{( \log{N})^{1-A/2r}}.
\end{align*}
There exists an integer 
$$k\le (\log{N})^{d(4r^2+4rA)},$$
and a multiplicative character $\psi \mod{k}$ such that 
\begin{align*}
\sum_{1\le n \le N}f(n)e(F(n))\ll (\log{N})^{4r^2+4rA}\max_{u\le N} \left| \sum_{n\le u}\psi(n)f(n)\right|.
\end{align*}
\end{corollary}

Corollary~\ref{cor:large} implies one may bound character sums mixed by polynomials by reducing to pure character sums, we refer the reader to Section~\ref{app} for more precise results.
\begin{corollary}
\label{cor:Hild1}
Let $F(n)$ a polynomial of degree $d$ with real coefficients and $\chi$ a primitive characters modulo $q$. Suppose that $\delta,\varepsilon$ satisfy
$$\max_{\substack{k\ll (\log{q})^{100d^3} \\ \psi \mod{k} \\ u \le N}}\left|\sum_{n\le u}\psi(n)\chi(n)\right|\le Nq^{-\varepsilon} \quad \text{provided} \quad N\ge q^{\delta}.$$
 Then we have 
\begin{equation}
\label{eq:eqn-11sums}
\sum_{n \le N} \chi(n) e(F(n))\ll\frac{N}{(\log N)^{1-1/d(d+1)}}.
\end{equation}
\end{corollary}
In particular, the estimate~\eqref{eq:eqn-11sums} holds under the following conditions:
\begin{itemize} 
\item For arbitrarily small $\delta$ assuming the Generalised Riemann hypothesis
\item For $\delta=1/3$ and arbitrary integer $q$, which follows from the Burgess bound, see for example~\cite{I-K}.
\end{itemize}
Let $\varepsilon>0$ be small.  Enflo~\cite{Enflo} has previously established that 
$$\sum_{n \le N} \chi(n) e(F(n))\ll N^{1-\delta}, \quad \text{provided} \quad N\ge q^{1/4+\varepsilon},$$
and we refer the reader to~\cite{Chang,H-B} for quantitative improvements on Enflo's result.
 Chang~\cite{Chang1} has shown
$$\sum_{n \le N} \chi(n) e(F(n))\ll N^{1-\delta}, \quad \text{provided} \quad N\ge q^{\varepsilon},$$
provided $q$ is suitably smooth/powerful. 

Corollary~\ref{cor:Hild1} provides some new instances where one may bound mixed character sums nontrivially. We refer the reader to Section~\ref{app} for more details.
\subsection{Acknowledgements}
We would like to thank Emmanuel Kowalski for his helpful comments and suggestions. We would also like to thank Alberto Perelli for a useful conversation and his helpful comments.
\subsection{Funding}
Part of the research done by Matteo Bordignon was supported by OP RDE project No.
CZ.$02.2.69/0.0/0.0/18\_053/0016976$ International mobility of research,
technical and administrative staff at the Charles University, part by an Australian Mathematical Society Lift-off Fellowship and part was done while the author was a PhD student of the University of New South Wales Canberra.

The third author would like to acknowledge the support of the Max Planck Institute for Mathematics and the Australian Research Council (DE220100859).
\section{Preliminary results}
\label{s1}
\subsection{Reduction to bilinear forms}
We proceed in a similar fashion to~\cite[Section~2]{M-V}, which reduces the proof of Theorem \ref{theo:main} to bounding bilinear 
forms under the hyperbola.
\begin{lemma}
\label{lem:bilinear}
Let $f$ be a multiplicative function satisfying the conditions of Theorem \ref{theo:main}, $g$ be any real valued function and let $\epsilon > 0$.  Then for any integer $N$ we have 

\begin{align}
\label{lem:eq}
\begin{split}
&\left|\sum_{1\le n \le N}f(n)e(g(n))\right|\ll \\ &\frac{N}{\log{N}^{1-\epsilon }} 
+\frac{1}{\log{N}}\left|\sum_{1\le np \le N} f(n)f(p)(\log{p}) e(g(np))\right|.
\end{split}
\end{align}
\end{lemma}
\begin{proof}

We follow the argument from~\cite[Section~2]{M-V} with some modifications to deal with the condition
$$\sum_{n\le N}|f(n)|^2\ll N(\log{N})^{A}.$$

Consider 
\begin{equation*}
S=\sum_{n\le N} f(n)e(g(n))\log N/n.
\end{equation*}
Since 
$$S=\log{N}\sum_{n\le N}f(n)e(g(n))-\sum_{n\le N}f(n)(\log{n})e(g(n)),$$
it is sufficient to show 
\begin{align}
\label{eq:toshow1}
S\ll_{\epsilon} N(\log{N})^{\epsilon},
\end{align}
and 
\begin{align}
\label{eq:toshow2}
\sum_{n\le N}f(n)&(\log{n})e(g(n))\ll_{\epsilon} \\ & \left|\sum_{1\le np \le N} f(n)f(p)(\log{p}) e(g(np))\right|+N(\log{N})^{\epsilon}. \nonumber
\end{align}
Let $r$ be a large real number and apply H\"{o}lder's inequality,~\eqref{eq:fell1} and~\eqref{eq:fell2} to get 
\begin{align*}
|S|^{2r}&\ll \left(\sum_{n\le N}(\log{N/n})^{2r} \right)\left(\sum_{n\le N}|f(n)|^{2} \right)\left(\sum_{n\le N}|f(n)| \right)^{2r-2} \\
&\ll \left(\sum_{n\le N}(\log{N/n})^{2r} \right) N^{2r-1}(\log{N})^{A}.
\end{align*}
Since 
\begin{align*}
\sum_{n\le N}(\log{N/n})^{2r}&\ll \sum_{\substack{j \\ 1\le 2^{j}\le N}}j^{2r}\sum_{N/2^{j+1}\le n \le  N/2^{j-1}}1 \\ 
&\ll N\sum_{j\ge 1}\frac{j^{2r}}{2^{j}}\ll N,
\end{align*}
we obtain~\eqref{eq:toshow1} after taking $r$ sufficiently large.
Since 
$$\log n=\sum_{d|n} \Lambda (d),$$ we have 
\begin{align*}
\sum_{n\le N}f(n)(\log{n})e(g(n))&=\sum_{nm\le N}\Lambda(m)f(nm)e(g(nm)) \\ & =\sum_{mn\le N}\Lambda(m)f(n)f(m)e(g(nm)) \\& +O\left(\sum_{mn\le N}\Lambda(m)|f(nm)-f(n)f(m)| \right).
\end{align*}
Using~\eqref{eq:fprime} and~\eqref{eq:fell1}
\begin{align*}
\sum_{mn\le N}\Lambda(m)f(n)f(m)e(g(nm))&=\sum_{pn\le N}(\log{p})f(p) f(n)e(g(nm)) \\ &+O\left(N\sum_{k\ge 2}\sum_{p^{k}\le N}\frac{|f(p^{k})|}{p^{k}}. \right) \\ 
&=\sum_{pn\le N}(\log{p})f(p) f(n)e(g(nm))\\ &+O(N),
\end{align*}
since the Cauchy-Schwarz inequality,~\eqref{eq:fell2} 
\begin{align*}
\sum_{k\ge 2}\sum_{p^{k}\le N}\frac{|f(p^{k})|}{p^{k}}\ll \left(\sum_{k\ge 2}\sum_{p}\frac{1}{p^{0.8k}} \right)\left(\sum_{n}\frac{|f(n)|^2}{n^{1.2}} \right)=O(1).
\end{align*}
Hence it remains to show 
\begin{align}
\label{eq:last-step1}
\sum_{mn\le N}\Lambda(m)|f(nm)-f(n)f(m)| =O(N(\log{N})^{\epsilon}).
\end{align}
From~\eqref{eq:fell1}
\begin{align*}
&\sum_{mn\le N}\Lambda(m)|f(nm)-f(n)f(m)| \\& \ll  \sum_{k\ge 1}\sum_{p^{k}\le N}\sum_{\substack{n\le N/p^{k} \\ p|n}}|f(p^{k})||f(n)|+|f(p^{k}n)| \\
&\ll \sum_{k,j\ge 1}\sum_{p^{k+j}\le N}\sum_{\substack{n\le N/p^{k+j} \\ (n,p)=1}}(|f(p^{k+j})|+|f(p^{k})||f(p^{j})|)|f(n)| \\
&\ll N\sum_{k,j\ge 1}\sum_{p}\frac{|f(p^{k})||f(p^{j})|+|f(p^{k+j})|}{p^{k+j}}.
\end{align*}
Applying the Cauchy-Schwarz inequality,~\eqref{eq:fell2} and partial summation to summation over $p,k,j$ as above, we establish~\eqref{eq:last-step1} which completes the proof.
\end{proof}
We require a generalisation of~\cite[Section~3]{M-V} for multiplicative functions $f$ satisfying~\eqref{eq:fell2}. 
\begin{lemma}
\label{lem:partition}
Let notation and conditions be as in Theorem~\ref{theo:main}. Suppose $s$ is a parameter and for each $0\le i \le \log_2{N}$ write
\begin{align}
\label{eq:Jidef}
J_i=\min\left(i+1,\lfloor \log_2{N}\rfloor-i+1,\lfloor \log_2(64N/s)/2\rfloor \right).
\end{align}
Define rectangles 
\begin{align}
\label{eq:Ri}
\cR_i=(0,2^{i}]\times \left(\frac{N}{2^{i+1}},\frac{N}{2^{i}}\right], \quad 0\le i \le \log_2{N},
\end{align}
\begin{align}
\label{eq:Rijk}
\begin{split}
\cR_{i,j,k}=&\left(\frac{2^{i+j}}{k},\frac{2^{i+j+1}}{2k-1}\right]\times \left(\frac{(k-1)N}{2^{i+j}},\frac{(2k-1)N}{2^{i+j+1}} \right], \\
 &0\le i \le \log_2{N}, \quad 1\le j \le J_i, \quad 2^{j-1}\le k \le 2^{j}.
\end{split}
\end{align}
Then each $\cR_{i,j,k}$ is a rectangle of the form $(P,P']\times (N,N'],$ with $P,P',N,N'$ satisfying 
\begin{align*}
P'-P\ge \frac{1}{4}, \quad N'-N\ge \frac{1}{4}, \quad (P'-P)(N'-N)\gg s.
\end{align*}
Let $\cE$ denote the set of points  $(p,n)$ satisfying $1\le pn\le N$ which do not lie in any of the rectangles~\eqref{eq:Ri} or~\eqref{eq:Rijk}. Then, for any fixed $\epsilon >0$, we have 
\begin{align*}
\sum_{(p,n)\in \cE} f(p)(\log{p})&f(n)e(g(pn)) \ll \\ &\ll  (\log N)^{\epsilon}(N+(Ns)^{1/2}\log({2N/s})\log{s}).
\end{align*}

\end{lemma}
\begin{proof}

Our proof is similar to that of~\cite[Section~3]{M-V} with some minor modifications. We first partition 
\begin{align*}
& \sum_{(p,n)\in \cE}f(p)(\log{p})f(n)e(g(pn))\ll \sum_{(p,n)\in \cE_1}|f(p)||f(n)|(\log{p}) \\ &+\sum_{(p,n)\in \cE_2}|f(p)||f(n)|(\log{p})+\sum_{(p,n)\in \cE_3}|f(p)||f(n)|(\log{p}),
\end{align*}
where
\begin{align*}
\cE_1=\{ (p,n)\in \cE : ~ pn\le N, ~ \frac{N}{2^{i+1}}<n\le \frac{N}{2^{i}}, ~ J_i> i+1 \},
\end{align*}
\begin{align*}
\cE_2=\{ (p,n)\in \cE : ~ pn\le N, ~ \frac{N}{2^{i+1}}<n\le \frac{N}{2^{i}}, ~ J_i=\lfloor \log_2{N}\rfloor-i+1 \},
\end{align*}
\begin{align*}
\cE_3=\{ (p,n)\in \cE : ~ pn\le N, ~ \frac{N}{2^{i+1}}<n\le \frac{N}{2^{i}},  ~ J_i=\lfloor \log_2(64N/s)/2\rfloor \}.
\end{align*}
Consider first $\cE_1$. By H\"{o}lder's inequality and~\eqref{eq:fprime}
\begin{align*}
&\left(\sum_{(p,n)\in \cE_1}|f(p)||f(n)|(\log{p})\right)^{2r}\ll \\ & \left(\sum_{(p,n)\in \cE_1}|f(n)| \right)^{2r-2}\left(\sum_{(p,n)\in \cE_1}|f(n)|^2 \right)\left(\sum_{(p,n)\in \cE_1}(\log{p})^{2r} \right).
\end{align*}
For each prime $p$, 
\begin{align*}
\#\{ n \ : \ (p,n)\in \cE_1\} \ll \frac{N}{p^2},
\end{align*}
so that 
\begin{align*}
\sum_{(p,n)\in \cE_1}(\log{p})^{2r}\ll N\sum_{p}\frac{(\log{p})^{2r}}{p^2} \ll N,
\end{align*}
and for each $n$
$$\#\{ p \ : \ (p,n)\in \cE_1\} \ll 1.$$
Hence by~\eqref{eq:fell1} and~\eqref{eq:fell2}
\begin{align*}
\sum_{(p,n)\in \cE_1}|f(n)| &\ll \sum_{n\le N}|f(n)|\ll N, \\  \sum_{(p,n)\in \cE_1}|f(n)|^2& \ll \sum_{n\le N}|f(n)|^2\ll N(\log{N})^{A}.
\end{align*}
Taking $r$ sufficiently large, the above estimates combine to give
\begin{align*}
\sum_{(p,n)\in \cE_1}|f(p)||f(n)|(\log{p})\ll_{\epsilon} N(\log{N})^{\epsilon}.
\end{align*}
Consider next $\cE_2$. By the Cauchy-Schwarz inequality 
\begin{align*}
\left(\sum_{(p,n)\in \cE_2}|f(p)||f(n)|(\log{p})\right)^2\ll \sum_{(p,n)\in \cE_2}|f(n)|^2\sum_{(p,n)\in \cE_2}(\log{p})^2.
\end{align*}
If $(p,n)\in \cE_2$ then $n\ll N^{1/2}$ and for fixed $n$, there exists some $H$ such that 
$$\{ p \ ; \ (p,n)\in \cE_2\} \subseteq [H,H+N/n^2].$$
Hence by the Brun-Titchmarsh theorem 
\begin{align*}
\sum_{(p,n)\in \cE_2}|f(p)||f(n)|(\log{p})\ll N\sum_{n\le N^{1/2}}\frac{|f(n)|^2}{n^2\log{(4N/n^2)}},
\end{align*}
which combined with~\eqref{eq:fell2} and partial summation 
\begin{align*}
\sum_{(p,n)\in \cE_2}|f(n)|^2\ll \frac{N}{\log{N}}.
\end{align*}
For each $p$
$$\#\{ n \ : (p,n)\in \cE_2\} \ll 1,$$
so that 
$$\sum_{(p,n)\in \cE_2}(\log{p})^2\ll \sum_{p\ll N}(\log{p})^2 \ll N\log{N}.$$
The above estimates combine to give 
\begin{align*}
\sum_{(p,n)\in \cE_2}|f(p)||f(n)|(\log{p})\ll N.
\end{align*}
Finally consider $\cE_3$. By H\"{o}lder's inequality 
\begin{align*}
&\left(\sum_{(p,n)\in \cE_3}|f(p)||f(n)|(\log{p})\right)^{2r}\ll  \\ & \left(\sum_{(N/q)^{1/2}\le n\le (Nq)^{1/2}}|f(n)|^{2r/(2r+1)}(\log{p})\right)^{2r-1}\\ & \times \left(\sum_{(N/q)^{1/2}\le p\le (Nq)^{1/2}}\frac{\log{p}}{p}\right).
\end{align*}
 If $(p,n)\in \cE_3$ then 
\begin{align*}
\left(\frac{N}{s}\right)^{1/2}\le p \le (Ns)^{1/2},
\end{align*}
and for each $p$, there exists some $H$ such that 
$$\#\{ n \ : \ (p,n)\in \cE_3 \}\subseteq [H,H+O((Ns)^{1/2}/p)],$$
and for each $n$ there exists some $H$ such that 
$$\#\{ p \ : \ (p,n)\in \cE_3 \}\subseteq [H,H+O((Ns)^{1/2}/n)].$$
Using the Brun-Titchmarsh inequality, we see that
\begin{align*}
&\left(\sum_{(p,n)\in \cE_3}|f(p)||f(n)|(\log{p})\right)^{2r}\ll  (Ns)^{1/2} \\ & \times \left(\sum_{(N/s)^{1/2}\le n\le (Ns)^{1/2}}|f(n)|^{2r/(2r+1)}\frac{\log{2N/n}}{n\log{2Ns/n^2}}\right)^{2r-1}\\ & \times \left(\sum_{(N/s)^{1/2}\le p \le (Ns)^{1/2}}\frac{\log{p}}{p}\right).
\end{align*}
 We have 
\begin{align*}
\sum_{(N/s)^{1/2}\le p \le (Ns)^{1/2}}\frac{\log{p}}{p}\ll \log{s},
\end{align*}
and H\"{o}lder's inequality combined with~\eqref{eq:fell1} and~\eqref{eq:fell2} give 
\begin{align*}
&\left(\sum_{(N/s)^{1/2}\le n\le (Ns)^{1/2}}|f(n)|^{2r/(2r+1)}\frac{\log{2N/n}}{n\log{2Ns/n^2}}\right)^{2r-1} \\ &
\ll (\log(2N/s))^{2r-1}\left(\sum_{(N/s)^{1/2}\le n\le (Ns)^{1/2}}\frac{|f(n)|}{n}\right)^{2r-2}\\ & \times \left(\sum_{(N/s)^{1/2}\le n\le (Ns)^{1/2}}\frac{|f(n)|^2}{n}\right) \\ 
& \ll (\log(2N/s))^{2r-1}(\log{N})^{A}(\log{s})^{2r-1}.
\end{align*}
Which after taking $r$ sufficiently large gives 
\begin{align*}
\sum_{(p,n)\in \cE_3}|f(p)||f(n)|(\log{p})\ll (Ns)^{1/2}(\log{N})^{\varepsilon}(\log{s})\log{(2N/s)}.
\end{align*}
from which the result follows.
\end{proof}
\section{Sums over bilinear forms}
\label{s33}
\subsection{The Vinogradov mean value theorem}
\label{s3}
 Given integers $r,d,V$ we let $J_{r,d}(V)$ count the number of solutions to the system of equations 
\begin{align*}
v_1^{j}+\dots-v_{2r}^{j}=0, \quad 1\le j\le d,
\end{align*}
with variables satisfying
$$1\le v_1,\dots,v_{2r}\le V.$$
We will use a consequence of Bourgain, Demeter and Guth's work on the Vinogradov mean value theorem, see~\cite[Section~5]{BDG}.
\begin{lemma}
\label{lem:vmvt}
Assume $d\ge 2$ and $r >d(d+1)$.
Then we have 
\begin{align*}
J_{r,d}(V)\ll_k V^{2r-d(d+1)/2}.
\end{align*}
\end{lemma}

Combining Lemma~\ref{lem:vmvt} with the fact that the Vinogradov system is translation invariant, we obtain the following. 
\begin{corollary}
\label{eq:cor76867}
Let $d\ge 2$, $r>d(d+1)$ be integers and $\cM(k), 1\le k \le K$ disjoint intervals satisfying 
\begin{align*}
\cM(k)=(M'(k),M''(k)], \quad M''(k)-M'(k)\le Y.
\end{align*}
Then the number of solutions to the system of equations 
\begin{align}
\label{eq:system1}
n_1^{j}+\dots-n_{2r}^{j}=0, \quad n_i\in \cM(k), \quad 1\le k \le K,
\end{align}
is bounded by 
$$O(KY^{2r-d(d+1)/2}).$$
\end{corollary}
We will also require an estimate for the number of solutions to the Vinogradov system with prime variables in translated intervals, here it will be fundamental that the intervals will not be 'too short'. 
To obtain such a result we need the following intermediary lemma that follows directly from \cite[Theorem 10]{Hua}, here appears clear why it is important that the intervals we work with are quite large compared to their starting point. We use $L=\log P$.
\begin{lemma}
\label{lem:Plarge}
Let $0<Q\le c_1(k)L^{\sigma_1}$, $X\gg 1$ and
$$S(X,P)=\sum_{\substack{X<p\le X+P \\ p \equiv t \pmod Q}} e(f(p)) $$
in which
$$P\gg \frac{X}{(\log X)^M}, $$
for any $M\gg 1$ and
$$f(x)=\frac{h}{q}x^k+\alpha_1x^{k-1}+\cdots+\alpha_k, \quad (h,q)=1, $$
the number $\alpha $ being real. Suppose that $L^\sigma< q \le P^kL^{-\sigma}$. For arbitrary $\sigma_0>0$, when $\sigma \le 2^{6k}(\sigma_0+\sigma_1+1)$, we always have
$$|S(X,P)| \le c_2(k)\frac{P}{Q L^{\sigma_0-M}}. $$
\end{lemma}
This lemma allows us to prove the following estimate for the number of solutions to the Vinogradov system with prime variables in large translated intervals.
\begin{lemma}
\label{lem:vmvtP1}
Let $d\ge 2$ be an integer and $X, Y\gg 1$. If $r>d(d+1)$,  the number of solutions to the equation 
\begin{align*}
&p_1^{j}+\dots-p_{2r}^{j}=0, \quad 1\le j \le d, \quad Y\le p_i\le Y+X, \\ & X\gg\frac{Y}{(\log Y)^M}, \quad p_i \quad \text{prime},
\end{align*}
for any $M\gg1$, is bounded by 
\begin{align*}
O\left(\frac{X^{2r-d(d+1)/2}}{(\log{X})^{2r}}\right).
\end{align*}
\end{lemma}
\begin{proof}
The result follows in a straightforward way from~\cite[Theorem~16]{Hua} re-defining in part 1) 
$$S(\alpha_k,\cdots,\alpha_1)=\sum_{Y\le p \le Y+X}e(f(p)), \quad f(x)=\alpha_kx^k+\cdots+\alpha_1 x,$$
and introducing the the two following changes which account for the translation in the set of primes and optimize the range of $r$ in Hua's result.
In part 3) of the proof we use our Lemma \ref{lem:Plarge} instead of \cite[Theorem~10]{Hua}, this is possible as $s_1$ in \cite[Lemma 10.8]{Hua} is of arbitrary size. While in part 5) of the proof we need to substitute~\cite[Theorem~15]{Hua} with~\cite[Theorem~1]{BDG}. Doing this we need to be careful to isolate, in~\cite[pag. 144]{Hua}, $\left|S(\alpha_k,\cdots,\alpha_1) \right|^{d(d+1)-\epsilon} $ instead of $\left|S(\alpha_k,\cdots,\alpha_1) \right|^{s_0-1}$ and, in~\cite[pag. 145]{Hua}, $\left|S(\alpha_k,\cdots,\alpha_1) \right|^{s-\epsilon} $ instead of $\left|S(\alpha_k,\cdots,\alpha_1) \right|^{s-1}$. Here $\epsilon >0$ it is such that $s=d(d+1)+2\epsilon$. 
\end{proof}
\subsection{Bounding bilinear forms}
\label{s31}
It is well known that one may use Lemma~\ref{lem:vmvt} to estimate bilinear forms with Weyl sums over rectangles. We next show how one may obtain sharper results by averaging bilinear forms over a family of disjoint rectangles. 
\begin{lemma}
\label{lem:bilinear1}
Let $F\in \R[X]$ be a polynomial of degree $d\ge 2$  of the form 
$$F(x)=\alpha_d x^{d}+\dots +\alpha_1 x.$$
 Let $M,N$ be integers and $\alpha(n),\beta(m)$ two sequences of complex numbers satisfying 
\begin{align*}
|\beta(m)|\le 1,
\end{align*}
and
\begin{align}
\label{eq:condi}
\sum_{n\le N} |\alpha (n)| \ll N, \quad \sum_{n\le N} |\alpha (n)|^2 \ll N(\log{N})^A,
\end{align}
for $A \ge 0$.
For $1\le k \le K$ let 
\begin{align}
\label{eq:Rk}
\cR(k)=\cL(k)\times \cM(k),
\end{align}
 be a rectangle of the form 
$$\cL(k)=(Q'(k),Q''(k)], \quad \cM(k)=(M'(k)\times M''(k)],$$
such that $\cL(k)\subseteq (0,Q)$, with $Q\gg 1$, are disjoint and satisfy 
\begin{equation}
\label{eq:cond1}
Q''(k)-Q'(k)\le X,
\end{equation}
or
\begin{equation}
\label{eq:cond2}
Q''(k)-Q'(k)= X \gg \frac{ Q'(k)}{(\log Q'(k))^M}
\end{equation}
and $\cM(k)\subseteq (0,M]$ are disjoint and satisfy 
$$M''(k)-M'(k)\le Y, \quad M''(k)\le 2M'(k),$$
and $K\ll M$.
Let $I$ denote the sum 
\begin{align}
\label{eq:Idef}
I=\sum_{k=1}^{K}\sum_{\substack{(p,n)\in \cR(k)}}\alpha(n)\beta(p)e(F(np)).
\end{align}
Let $R\ge 1$ and suppose there exists $q\le R$ and $1\le \ell \le d$ such that 
\begin{align}
\label{eq:alassump}
\left|\alpha_{\ell}-\frac{a}{q}\right|\le \frac{1}{qR},
\end{align}
for some $(a,q)=1$. For $r>d(d+1)$, we have

\begin{align*}
I^{4r^2}& \ll m(X)(\log{M})^{2rA}M^{4r^2}\left(\frac{X}{\log{X}}\right)^{4r^2} \left(\frac{M}{Y}\right)^{d(d-1)/2}\\ &\left(\frac{Q}{X}\right)^{d(d-1)/2}  \left(\frac{q}{XQ^{\ell-1}YM^{\ell-1}}+\frac{1}{XQ^{\ell-1}}+\frac{1}{YM^{\ell-1}}+\frac{1}{q}\right),
\end{align*}
where
$$m(x)=\begin{cases*}
(\log{x})^{4r}& \textit{when} \eqref{eq:cond1} \textit{holds},\\
1 & \textit{when} \eqref{eq:cond2} \textit{holds}.
\end{cases*} $$
\end{lemma}

\begin{proof}

Recall that
$$F(x)=\alpha_d x^d+\dots+\alpha_1 x,$$
and define
\begin{align}
\label{eq:Sums1}
S_k=\sum_{\substack{(p,n)\in \cR(k)}}\alpha(n)\beta(p)e(F(np)),
\end{align}
so that 
\begin{align}
\label{eq:IS}
I=\sum_{k=1}^{K}S_k.
\end{align}
Fix some $1\le k \le K$ and consider~\eqref{eq:Sums1}. Recalling~\eqref{eq:Rk}, by H\"{o}lder's inequality, for any integer $r\ge 1$
\begin{align*}
&S_{k}\le \\ & \left(\sum_{n\in \cM(k)}|\alpha(n)|^{2r/(2r-1)}\right)^{1-1/2r}\left(\sum_{n\in \cM(k)}\left|\sum_{p\in \cL(k)}\beta(p)e(F(np))\right|^{2r}\right)^{1/2r}.
\end{align*}
After another application of H\"{o}lder's inequality 
\begin{align*}
&I^{2r}\le \\
&\left(\sum_{1\le k \le K}\sum_{n\in \cM(k)}|\alpha(n)|^{2r/(2r-1)} \right)^{2r-1}\sum_{1\le k \le K}\sum_{n\in \cM(k)}\left|\sum_{p\in \cL(k)}\beta(p)e(F(np))\right|^{2r}.
\end{align*}
Hence from assumptions on the $\cM(k)$
\begin{align*}
I^{2r}\ll \left( \sum_{n \le M} |\alpha (n)|^{2r/(2r-1)} \right)^{2r-1}\sum_{1\le k \le K}\sum_{n\in \cM(k)}\left|\sum_{p\in \cL(k)}\beta(p)e(F(np))\right|^{2r},
\end{align*}
which combined with~\eqref{eq:condi} implies that 
\begin{align*}
I^{2r}\ll  M^{2r-1}(\log{M})^{A } \sum_{1\le k \le K}\sum_{n\in \cM(k)}\left|\sum_{p\in \cL(k)}\beta(p)e(F(np))\right|^{2r}.
\end{align*}
Write $\lambda=(\lambda_1,\dots,\lambda_d)$ and let $J_k(\lambda)$ count the number of solutions to 
$$p_1^{j}+\dots-p_{2r}^j=\lambda_j, \quad 1\le j \le d, \quad p_i\in \cL(k), \quad p_i \ \ \text{prime}.$$
Note by assumptions on $\cL(k)$, if $p_1,\dots,p_{2r}\in \cL(k)$ satisfy 
$$p_1^{j}+\dots-p_{2r}^j=\lambda_j,$$
then 
$$\lambda_j\ll XQ^{j-1}.$$
Expanding the $2r$-th power and interchanging summation gives
\begin{align*}
I^{2r}\ll& M^{2r-1}(\log{M})^{A}\\ & \times \sum_{1\le k \le K}\sum_{\substack{\lambda \\ \lambda_j \ll XQ^{j-1}}} J_k(\lambda)\left|\sum_{n\in \cM(k)}e(\alpha_d\lambda_d n^{d}+\dots+\alpha_1\lambda_1 n)\right|.
\end{align*}
After another application of H\"{o}lder's inequality, we get  
\begin{align*}
& I^{4r^2}\ll M^{4r^2-2r}(\log{M})^{2rA}\left(\sum_{1\le k \le K}\sum_{\lambda} J_k(\lambda)\right)^{2r-2}\left(\sum_{1\le k \le K}\sum_{\lambda} J_k(\lambda)^2\right) \\ 
& \quad \times \sum_{1\le k \le K}\sum_{\substack{ \lambda_j \ll XQ^{j-1}}}\left|\sum_{n\in \cM(k)}e(\alpha_d\lambda_d n^{d}+\dots+\alpha_1\lambda_1 n)\right|^{2r}.
\end{align*}

We have 
\begin{align*}
\sum_{1\le k \le K}\sum_{\lambda} J_k(\lambda)\le \sum_{1\le k \le K}|\{ p_1,\dots,p_{2r}\in \cL(k)\}|.
\end{align*}

Since for each $1\le k \le K$,  $\cL(k)$ is an interval of length at most $X$, the Brun-Titshmarsh theorem implies 
\begin{align*}
\sum_{1\le k \le K}\sum_{\lambda} J_k(\lambda)\ll \frac{KX^{2r}}{(\log{X})^{2r}}.
\end{align*}
The term
\begin{align*}
\sum_{1\le k \le K}\sum_{\lambda} J_k(\lambda)^2,
\end{align*}
counts the number of solutions to the system of equations 
\begin{align}
\label{eq:123456-1}
p_1^{j}+\dots-p_{4r}^{j}=0, \quad 1\le j \le d, \quad p_i\in \cL(k), \ \ \text{ $p_i$ prime}, \quad 1\le k \le K.
\end{align}
Ignoring the condition that $p_i$ is prime, by Corollary~\ref{eq:cor76867} and Lemma \ref{lem:vmvtP1},
\begin{align*}
\sum_{1\le k \le K}\sum_{\lambda} J_k(\lambda)^2\ll K m_1(X), 
\end{align*}
where
$$m_1(x)=\begin{cases*}
X^{4r-d(d+1)/2}& \textit{when} \eqref{eq:cond1} \textit{holds},\\
\frac{X^{4r-d(d+1)/2}}{(\log X)^{4r}} & \textit{when} \eqref{eq:cond2} \textit{holds}.
\end{cases*} $$
since $r>d(d+1)$. Combined with the observations above, this implies 
\begin{align}
\label{eq:II0}
& I^{4r^2}\ll m(X) M^{4r^2-2r}(\log{M})^{2rA}K^{2r-1}\left(\frac{X}{\log{X}}\right)^{4r^2}\frac{I_0}{X^{d(d+1)/2}},
\end{align}
where 
$$I_0=\sum_{1\le k \le K}\sum_{\substack{ \lambda_j \ll XQ^{j-1}}}\left|\sum_{n\in \cM(k)}e(\alpha_d\lambda_d n^{d}+\dots+\alpha_1\lambda_1 n)\right|^{2r}.$$
Let 
$$S(x)=\left(\frac{\sin{\pi x}}{\pi x}\right)^2,$$
so that 
\begin{align}
\label{eq:Shat}
\widehat S(x)=\max\{0,1-|x|\},
\end{align}
and 
$$S(x) \gg 1 \quad \text{if} \quad |x|\le \frac{1}{4}.$$
There exists a constant $c_0$ such that 
\begin{align*}
&I_0\ll \\ & \sum_{1\le k \le K}\sum_{\substack{ \lambda_j \ll XQ^{j-1} \\  j\neq \ell}}\sum_{\lambda_{\ell}\in \Z}S\left(\frac{c_0\lambda_{\ell}}{XQ^{\ell-1}}\right)\left|\sum_{n\in \cM(k)}e(\alpha_d\lambda_d n^{d}+\dots+\alpha_1\lambda_1 n)\right|^{2r}.
\end{align*}

Expanding the $2r$-th power and interchanging summation 
\begin{align*}
&I_0\ll \\ &\sum_{\substack{1\le k \le K\\ \mu_j\ll YN^{j-1}}}L_k(\mu)\prod_{\substack{j=1 \\ j\neq \ell }}^{d}\left|\sum_{\lambda\ll XQ^{j-1}}e(\alpha_j \lambda \mu_j)\right|\left|\sum_{\lambda \in \Z}S\left(\frac{c_0\lambda_{\ell}}{XQ^{\ell-1}}\right)e(\alpha_\ell \lambda \mu_\ell)\right|,
\end{align*}
with $\mu=(\mu_1,\dots,\mu_d)$ and $L_k(\mu)$ counts the number of solutions to 
\begin{align*}
n^{j}_1+\dots-n_{2r}^{j}=\mu_j, \quad n_i\in \cM(k), \quad 1\le k \le K.
\end{align*}
Using that 
$$L_k(\mu)\le L_k(0),$$
and applying Corollary~\ref{eq:cor76867}, we obtain

\begin{align*}
& I_0\ll KY^{2r-d(d+1)/2} \\ & \sum_{\substack{\mu \\ \mu_j \ll YM^{j-1}}}\prod_{\substack{j=1 \\ j\neq \ell }}^{d}\left|\sum_{\lambda \ll XQ^{j-1}}e(\alpha_j \lambda \mu_j)\right|\left|\sum_{\lambda \in \Z}S\left(\frac{c_0\lambda}{Xq^{\ell-1}}\right)e(\alpha_\ell \lambda \mu_\ell)\right|.
\end{align*}
Combining the inequality above with~\eqref{eq:II0} we have 
\begin{align}
\label{eq:II1}
& I^{4r^2}\ll m(X)(\log{M})^{2rA}M^{4r^2}\left(\frac{X}{\log{X}}\right)^{4r^2}\frac{I_1}{(XY)^{d(d+1)/2}},
\end{align}
where 

\begin{align*}
I_1=\sum_{\substack{\mu_j \ll YM^{j-1}}}\prod_{\substack{j=1 \\ j\neq \ell }}^{d}\left|\sum_{\lambda \ll XQ^{j-1}}e(\alpha_j \lambda \mu_j)\right|\left|\sum_{\lambda \in \Z}S\left(\frac{c_0\lambda}{Xq^{\ell-1}}\right)e(\alpha_\ell \lambda \mu_\ell)\right|
\end{align*}
In $I_1$, we bound every term trivially except the one with index $\ell$ to obtain 
\begin{align}
\label{eq:I1I1}
I_1&\ll \left(\frac{Y}{M}\right)^{d}\left(\frac{X}{Q}\right)^{d}(QM)^{d(d+1)/2}\frac{1}{XQ^{\ell-1}}\frac{1}{YM^{\ell-1}} \\ &\sum_{\mu\ll YM^{\ell-1}}\left|\sum_{\lambda \in \Z}S\left(\frac{c_0\lambda}{Xq^{\ell-1}}\right)e(\alpha_\ell \lambda \mu_\ell)\right|.
\end{align}
By Poisson summation
\begin{align*}
\sum_{\lambda \in \Z}S\left(\frac{c_0\lambda_{\ell}}{XQ^{\ell-1}}\right)e(\alpha_\ell\mu \lambda)=\frac{XQ^{\ell-1}}{c_0}\sum_{\lambda \in \Z}\widehat S\left(\frac{XQ^{\ell-1}(\lambda-\alpha_{\ell}\mu)}{c_0}\right),
\end{align*}
and hence from~\eqref{eq:Shat}
\begin{align*}
&\sum_{\mu\ll YM^{\ell-1}}\left|\sum_{\lambda \in \Z}S\left(\frac{c_0\lambda_{\ell}}{XQ^{\ell-1}}\right)e(\alpha_\ell\mu \lambda)\right| \\ &  \quad \quad \quad \quad \quad \ll XQ^{\ell-1}\sum_{\mu\ll YM^{\ell-1}}\sum_{\lambda \in \Z}\max\left\{0,1-\left|\frac{XQ^{\ell-1}(\lambda-\alpha_{\ell}\mu)}{c_0} \right| \right\} \\ & 
\quad \quad \quad \quad \quad \ll XQ^{\ell-1}\left|\left\{ \mu \ll YM^{\ell-1} \ : \ \|\alpha_{\ell}\mu\|\le \frac{100c_0}{XQ^{\ell-1}} \right\}\right|.
\end{align*}
There exists some real number $\beta$ such that 
\begin{align*}
&\left|\left\{ \mu \ll YM^{\ell-1} \ : \ \|\alpha_{\ell}\mu\|\le \frac{100c_0}{XQ^{\ell-1}} \right\}\right| \\ & \quad \quad \quad \ll \left(1+\frac{YM^{\ell-1}}{q}\right)\left|\left\{ 0\le \mu \le q \ : \ \|\alpha_{\ell}\mu+\beta\|\le \frac{100c_0}{XQ^{\ell-1}} \right\}\right|.
\end{align*}
Recalling~\eqref{eq:alassump}
\begin{align*}
&\left|\left\{ \mu \ll YM^{\ell-1} \ : \ \|\alpha_{\ell}\mu\|\le \frac{100c_0}{XQ^{\ell-1}} \right\}\right| \\ &  \quad \quad \ll \left(1+\frac{YM^{\ell-1}}{q}\right)\left|\left\{ 0\le \mu \le q \ : \ \|\frac{a\mu}{q}+\beta\|\le \frac{100c_0}{XQ^{\ell-1}}+\frac{1}{R} \right\}\right|,
\end{align*}
which implies that 
\begin{align*}
&\left|\left\{ \mu \ll YM^{\ell-1} \ : \ \|\alpha_{\ell}\mu\|  \le \frac{100c_0}{XQ^{\ell-1}} \right\}\right|\\ & \quad \quad \quad \quad \quad \quad \ll \left(1+\frac{YM^{\ell-1}}{q} \right)\left(1+\frac{q}{XQ^{\ell-1}}\right),
\end{align*}
and hence 
\begin{align*}
&\sum_{\mu\ll YM^{\ell-1}}\left|\sum_{\lambda \in \Z}S\left(\frac{c_0\lambda}{Xq^{\ell-1}}\right)e(\alpha_\ell \lambda \mu_\ell)\right|\ll \\ & XQ^{\ell-1}YM^{\ell-1}\left(\frac{1}{YM^{\ell-1}}+\frac{1}{q} \right)\left(1+\frac{q}{XQ^{\ell-1}}\right).
\end{align*}
Combined with~\eqref{eq:II1} and~\eqref{eq:I1I1} gives 
\begin{align*}
& I^{4r^2}\ll \\ & m(X)(\log{M})^{2rA}M^{4r^2}\left(\frac{X}{\log{X}}\right)^{4r^2} \left(\frac{M}{Y}\right)^{d(d-1)/2}\left(\frac{Q}{X}\right)^{d(d-1)/2} \\ & \left(\frac{q}{XQ^{\ell-1}YM^{\ell-1}}+\frac{1}{XQ^{\ell-1}}+\frac{1}{YM^{\ell-1}}+\frac{1}{q}\right),
\end{align*}
which completes the proof.
\end{proof}
\section{Proof of Theorem \ref{theo:main}} 
\label{s4}
 We apply Lemma~\ref{lem:bilinear} and Lemma~\ref{lem:partition} with

$$s=\frac{q^{1/\ell}}{(\log{N})^{4}},$$
 to get 
\begin{align*}
\sum_{1\le n \le N}f(n)e(F(n))\ll & \frac{N}{\log{N}^{1+o(1)}}+(Nq^{1/\ell})^{1/2}+\sum_{0\le i \le \log_2{N}}\frac{i}{\log{N}}S_i\\ & +\sum_{0\le i\le \log_2{N}}\frac{i}{\log{N}}\sum_{1\le j \le J_i}S_{i,j},
\end{align*}
where 
\begin{align*}
S_i=\sum_{(p,n)\in \cR_i}\frac{\log{p}}{i}f(p)f(n)e(F(pn)),
\end{align*}
and 
\begin{align*}
S_{i,j}=\sum_{2^{j-1}\le k \le 2^{j}}\sum_{(p,n)\in \cR_{i,j,k}}\frac{\log{p}}{i}f(p)f(n)e(F(pn)).
\end{align*}
Note if $(p,n)\in \cR_i$ or $(p,n)\in \cR_{i,j,k}$ then 
$$\frac{\log{p}}{i}\ll 1.$$
By Lemma~\ref{lem:bilinear1}, observing that in this case condition \eqref{eq:cond2} holds,
\begin{align*}
S_i&\ll (\log{N})^{A/2r}\frac{N}{i}\left(\frac{q}{N^{\ell}}+\frac{1}{2^{\ell i}}+\frac{2^{\ell i }}{N}+\frac{1}{q}\right)^{1/4r^2},
\end{align*}
which implies that 
\begin{align}
\label{eq:Sib}
\sum_{0\le i \le \log_2{N}}\frac{i}{\log{N}}S_i\ll \frac{N}{(\log{N})^{1-A/2r}}\left(1+\log{N}\left(\frac{q}{N^{\ell}}+\frac{1}{q}\right)^{1/4r^2}\right).\end{align}

Consider next the sums $S_{i,j}$. We apply Lemma~\ref{lem:bilinear1} with parameters 
\begin{align*}
K=2^{j-1}, \quad M=\frac{N}{2^{i}}, \quad Q=2^{i+1}, \quad X=2^{i-j+1}, \quad Y=32N2^{-i-j}.
\end{align*}
We first focus on the case when $j \ge \log i^c$, for a $c\gg1$. In this case we use Lemma~\ref{lem:bilinear1} with condition \eqref{eq:cond1}, which easily gives 
\begin{align*}
S_{i,j}\ll \frac{N}{i^{c_1}2^{j/8}}  \left(\frac{q}{N^{\ell}}+\frac{1}{2^{\ell i}}+\left(\frac{2^{i}}{N}\right)^{\ell}+\frac{1}{q}\right)^{1/4r^2},
\end{align*}
for $c_1\gg1$.
For fixed $i\le \log_2{N}$, recalling that $J_i$ is given by~\eqref{eq:Jidef}, we have 
\begin{align*}
\sum_{\log i^c\le j \le J_i}S_{i,j}\ll \frac{N}{i^{c_1}} \left(\frac{q}{N^{\ell}}+\frac{1}{2^{\ell i}}+\left(\frac{2^{i}}{N}\right)^{\ell}+\frac{1}{q}\right)^{1/4r^2},
\end{align*}
and hence 
\begin{align*}
&\sum_{0\le i\le \log_2{N}}\frac{i}{\log{N}}\sum_{\log i^c\le j \le J_i}S_{i,j} \\ &  \ll \frac{N}{\log{N}}\sum_{i\le \log_2{N}}\frac{1}{i^{c_1-1}}\left(\frac{q}{N^{\ell}}+\frac{1}{2^{\ell i}}+\left(\frac{2^{i}}{N}\right)^{\ell}+\frac{1}{q}\right)^{1/4r^2} \\
& \ll N\left(\frac{1}{\log{N}}+\frac{q}{N^{\ell}}+\frac{1}{q}\right).
\end{align*}
Denote $C = \frac{A}{2r}$. We then focus on the case when $j \le \log i^c$. In this case we use can use Lemma~\ref{lem:bilinear1} with condition \eqref{eq:cond2}, which easily gives 
\begin{align*}
S_{i,j}&\ll \frac{1}{(i-j+1)^{1-C}}\frac{N}{2^{j(1-d(d-1)/4r^2)}} \\ & \left(\frac{2^{2j}q}{N^{\ell}}+\frac{2^{j}}{2^{\ell i}}+2^{j}\left(\frac{2^{i}}{N}\right)^{\ell}+\frac{1}{q}\right)^{1/4r^2}.
\end{align*}
For fixed $i\le \log_2{N}$, we have 
\begin{align*}
\sum_{1\le j \le \log i^c}S_{i,j}\ll \frac{N}{i^{1-C}} \left(\frac{q}{N^{\ell}}+\frac{1}{2^{\ell i}}+\left(\frac{2^{i}}{N}\right)^{\ell}+\frac{1}{q}\right)^{1/4r^2},
\end{align*}
and hence 
\begin{align*}
&\sum_{0\le i\le \log_2{N}}\frac{i}{\log{N}}\sum_{1\le j \le \log i^c}S_{i,j} \\ &  \ll \frac{N}{\log{N}}\sum_{i\le \log_2{N}}i^{C}\left(\frac{q}{N^{\ell}}+\frac{1}{2^{\ell i}}+\left(\frac{2^{i}}{N}\right)^{\ell}+\frac{1}{q}\right)^{1/4r^2} \\
& \ll N\left(\frac{1}{(\log{N})^{1-C}}+\frac{q(\log{N})^{C}}{N^{\ell}}+\frac{(\log{N})^{C}}{q}\right).
\end{align*}
Combining the above estimates we complete the proof.
\section{Proof of Corollary~\ref{cor:large}}
Suppose 
$$F(x)=\alpha_1x+\dots+\alpha_d x^{d}.$$
By Dirichlet's theorem, for each $1\le \ell \le d$,  there exists integers $r_{\ell},s_{\ell}$ with $(r_{\ell},s_{\ell})=1$ and 
\begin{align*}
r_{\ell}\le \frac{N^{\ell}}{(\log{N})^{4r^2+4rA}},
\end{align*}
such that 
\begin{align*}
\left|\alpha_{\ell}-\frac{r_{\ell}}{s_{\ell}}\right|\le \frac{(\log{N})^{4r^2+4rA}}{q_{\ell}N^{\ell}}.
\end{align*}
By Theorem~\ref{theo:main}, we may suppose for each $1\le \ell \le d$
\begin{align*}
s_\ell\le (\log{N})^{4r^2+4rA}.
\end{align*}
By partial summation, we have
\begin{equation*}
\left| \sum_{n \le N} f(n) e(F(n))\right|\ll (\log{N})^{4r^2+4rA}\max_{u \le N} T(u) ,
\end{equation*}
with 
\begin{equation*}
T(u):=\sum_{n \le u} f(n) e(F_1(n)),
\end{equation*}
and
\begin{equation*}
F_1(x):=\frac{r_d}{s_d} x^d+\dots+\frac{r_1}{s_1} x.
\end{equation*} 
Defining $k:=lcm(s_d, \cdots,s_1)$, we have
\begin{equation*}
T(u)=\sum_{a \le k}e(F_1(a))S(a),
\end{equation*}
with
\begin{equation*}
S(a):=\sum_{\substack{n \le u\\ n \equiv a \bmod{k}}} f(n).
\end{equation*}
Let $d=(a,k)$ and write
\begin{equation*}
a'=\frac{a}{d}, \quad k'=\frac{k}{d},
\end{equation*}
so that 
\begin{equation*}
S(a)=\sum_{\substack{n \le u\\ n \equiv a' \bmod{k'} \\ n \equiv 0 \bmod{d}}} f(n)= \frac{1}{\phi(k')}\sum_{\psi \mod k'} \overline{\psi}(a')\sum_{\substack{n \le u\\ n\equiv 0 \bmod{d}}} \psi(n)f(n).
\end{equation*}
If $(d,q)\neq 1$ then $S(a)=0$. If $(d,q)=1$, then 
\begin{equation*}
|S(a)|\le \frac{1}{\phi(k')}\sum_{\psi \mod k'} \left|\sum_{n \le u/d} \psi(n)f(n)\right|.
\end{equation*}
Thus, observing that $$\operatorname{lcm}(s_d, \cdots,s_1)\le \prod_{1 \le i \le d} s_i\le (\log{N})^{d(4r^2+4rA)}$$ we conclude the proof.

\section{Short mixed character sums}
\label{app}
We next use Corollary~\ref{cor:large} to show how one may estimate short mixed character sums assuming GRH.
\begin{corollary}
\label{cor:Hild}
Let $F(n)$ a polynomial of degree $d$ with real coefficients given by
$$F(x)=\alpha_d x^d+\dots+\alpha_1 x$$
and taken $\chi$ a primitive characters modulo $q$. Then assuming GRH, uniformly for all primitive characters, polynomials $F$ of degree $d$ and 
\begin{equation}
\label{eq:cond}
\log{N}\ge B\log\log{q},
\end{equation}
for a suitable fixed constant $B$, we have
\begin{equation*}
\left| \sum_{n \le N} \chi(n) e(F(n))\right|\ll\frac{N}{(\log N)^{1-\varepsilon}},
\end{equation*}
for arbitrary $\varepsilon>0$.
\end{corollary}

We need the following lemma on convolution of Dirichlet characters, that follows from Theorem 2 in \cite{GS1} and Corollary 1.3 in \cite{HT}, remembering that there they have $u=\log x /\log y$.

\begin{lemma}
\label{lemma:1}
Let $\chi$ be a non-principal character modulo $q$. Assuming the Generalised Riemann hypothesis, for any $x$  such that
\begin{equation*}
\frac{\log x }{\log \log q}\ge B,
\end{equation*}
for a sufficiently large constant $B$,
we have
\begin{equation*}
\left|\sum_{n\le x} \chi(n) \right| \ll x(\log q)^{-(1+o(1) )\frac{\log x \log \left(\frac{\log x}{2\log \log q}\right)}{2(\log \log q)^2}},
\end{equation*}
\end{lemma}
We can thus prove Corollary \ref{cor:Hild} in a similar fashion to \cite[Lemma~3]{Hildebrand} and~\cite[Lemma~7]{Kerr}. 
\begin{proof}[Proof of Corollary \ref{cor:Hild}]
Clearly we may suppose 
\begin{equation*}
\left| \sum_{n \le N} \chi(n) e(F(n))\right|\gg\frac{N}{(\log N)^{1-\varepsilon}}.
\end{equation*}
For any $\psi$ with modulus smaller than or equal to $k:=lcm(s_d, \cdots,s_1)$, by Lemma \ref{lemma:1}, we have
\begin{equation*}
\cm{\max_{u\le N}}\left|\sum_{n \le u} \psi(n)\chi(n)\right| \ll N(\log q)^{-(1+o(1) )\frac{\log N \log \left(\frac{\log N}{2\log \log q}\right)}{2(\log \log q)^2}}.
\end{equation*} 
 Thus, by Corollary~\ref{cor:large} and~\eqref{eq:cond}, we have
\begin{align*}
\left| \sum_{n \le N} \chi(n) e(F(n))\right|&\ll N (\log{x})^{d^{5}}(\log q)^{-(1+o(1) )\frac{\log N \log \left(\frac{\log N}{2\log \log q}\right)}{2(\log \log q)^2}} \\ & \ll \frac{N}{(\log N)^{1-\varepsilon}},
\end{align*}
and this concludes the proof.
\end{proof}

\section{On the correlation between roots of polynomial congruences and polynomial values}
We now prove Theorem \ref{theorem2}.
Let $p \in \Z[x]$ be irreducible over $\Q$ of degree $e \ge 2$ and consider the ratios $v/n$, where $v$ are the roots of the polynomial $p$ modulo $n$
\begin{align*}
p(v) \equiv 0 \bmod{n}.
\end{align*}
Define the sequence $(g_k)_{k\ge 1}$ of these ratios so that the corresponding denominators are in ascending order.

Hooley~\cite[Theorem 1]{Hooley} proved  that $(g_k)_{k\ge 1}$ is equidistributed in $\R /\Z$. Even stronger than that, he showed 
\begin{align}
\label{ap1}
\lim_{x \rightarrow \infty} \frac{1}{x} \sum_{n \le x} \left| \sum_{ \substack{ v \in \Z / n\Z \\ p(v) \equiv 0 \bmod{n}}} e\left( \frac{hv}{n} \right) \right| = 0,
\end{align}
for every $h\in \Z\backslash \{0\}$.
To prove Theorem \ref{theorem2},  by the Weyl equidistribution criteria, it suffices to prove that 
for any $(h_1,h_2)\neq 0$,  we have
\begin{align*}
\sum_{n\le N} e\left(h_1 F(n) \right) \sum_{ \substack{ v \in \Z / n\Z \\ \cm{p}(v) \equiv 0 \bmod{n}}} e\left( \frac{h_2v}{n} \right) = o(N).
\end{align*}
If $h_2\neq 0$ this is true by equation (\ref{ap1}). Thus, the only cases that need to be considered are $h_2 = 0$ and $h_1 \neq 0$.  If we let
\begin{align}
\label{eq:rho}
\varrho(n) = | \{ v \in \Z/n\Z ; \cm{p}(v) \equiv{0}\bmod{n} \} |,
\end{align}
then the problem is therefore to show that
\begin{align*}
\sum_{n\le N} \varrho(n) e\left( F(n) \right) = o(N).
\end{align*}
We first prove some properties of $\varrho$.
\begin{lemma}
\label{lemma:rhoprop}
The function $\varrho$ is multiplicative and satisfies:
\begin{enumerate}
\item For all $\epsilon >0$, $\varrho(n)\le_{\epsilon} n^{\epsilon}$, for $n\rightarrow \infty$.
\item There exists a constant $\lambda >0$ such that $$ \sum_{n=1}^N \varrho(n) \sim   \lambda N. $$ 
\item There exists a constant $A\ge0$ such that$$ \sum_{n=1}^N \varrho(n)^2 \ll N (\log{N})^A,$$ for $N \ge 2$. 
\item There exists a constant $D\ge 1$ such that $$\varrho(mn) \le D^{\operatorname{disc}(f)} \varrho(m)\varrho(n),$$  for all $m,n \in \Z_{> 0}$.
\end{enumerate}
\begin{proof}
The Chinese Remainder Theorem implies that $\varrho$ is a multiplicative function. From Wirsing's theorem \cite[Satz 1]{Wirs} we conclude the second item. The proof of the third one is standard (see, for e.g \cite[Lemma 2.7]{pollack}).

To conclude the fourth item, we observe that
$\varrho(p^\alpha) = \varrho(p)$ for all $p$ primes, $p\nmid \operatorname{disc}(f)$ and $\alpha \in \Z_{\ge 1}$ and  $\varrho(p^\alpha) \le D$ for all $p$ primes and a constant $D>0$ (for a proof, see, for e.g.   \cite[Lemma 2.4]{pollack}).  We also note that if $p| \operatorname{disc}(f)$ and $\varrho(p^\alpha) \neq 0$ then $\varrho(p^\beta) \neq 0$ for all $\beta \le \alpha$ and we can conclude the result by factoring $m$ and $n$ in primes. Note this also implies the first item, since if $n=p_1^{\alpha_1}\dots p_k^{\alpha_k}$ then 
$$\varrho(n)\ll D^{k}\ll D^{\log{n}/{\log\log{n}}}.$$
\end{proof}
\end{lemma}

From now on, we fix $r^\ast = 2\max(d(d+1),A)$ and  $B=4{r^\ast}^2 +4rA+1$ and the constants $\lambda$, $D$ and $A$ as in Lemma~\ref{lemma:rhoprop}.  

 From Dirichlet's Theorem, we can find  $a_i$ and $1 \le q_i \le \frac{N^i}{(\log{N})^B}$ coprime integers satisfying
\begin{equation}
\label{eq:aprox}
\left| \alpha_i - \frac{a_i}{q_i} \right| \le \frac{ (\log{N})^B}{q_iN^i},
\end{equation} 
for $1 \le i \le d$.  We also let $$q=\operatorname{lcm}(q_1, \ldots, q_d).$$
\begin{prop}
\label{prop:qbig}
If there exists an $l \in \{1, \ldots, d\}$ such that $q_l$ satisfies $q_l \ge (\log{N})^B$ then
\begin{equation*}
\sum_{n=1}^N \varrho(n)e(F(n)) \ll \frac{N}{(\log{N})^{1/2}}.
\end{equation*}
\end{prop}
\begin{proof}
This follows directly from Theorem \ref{theo:main} applied to $f=\rho$, $r = r^\ast$ and $R=N^{\ell}/(\log{N})^B$.
\end{proof}

We next focus on establishing the following result.

\begin{prop}
\label{prop:main}
Suppose that $q_i< (\log{N})^B$, for all $1\le i \le d$. Then it holds that
\begin{equation*}
\sum_{n=1}^N\varrho(n)e(F(n)) \ll \frac{N}{q^{ \frac{2}{d(d+1) } -\epsilon}}.
\end{equation*}
\end{prop}

Combining Proposition~\ref{prop:main} with Proposition \ref{prop:qbig} one may deduce Theorem \ref{theorem2}. Indeed, since at least one of the $\alpha_i$ is irrational, as $N\rightarrow \infty$, the integers $q_i$ in~\eqref{eq:aprox} must satisfy $q_i\rightarrow \infty$ which implies  \begin{equation*}
\sum_{n=1} \varrho(n)e(F(n)) = o(N).
\end{equation*}
To prove Proposition \ref{prop:main} we will proceed with an algebraic approach.  We split the task in four subsections: first we make some reductions and provide a proof of the proposition, conditional to further analysis of a different sum. This new sum will be analysed using a Dirichlet series. In the next subsection, we write the Dirichlet series that we want to analyse in terms of better understood L-functions and in the third subsection we state some of its properties. At last, we gather all the results and conclude the proof.

\subsection{First reductions}
We assume  $q_i< (\log{N})^B$, for all $1\le i \le d$ and we set, for $r|q$, and $Q \in \Z[x]$
\begin{align*}
S_r(N, Q) = \sum_{\substack{n=1 \\ (n,q)=r}}^N \varrho(n) e( Q(n)),
\end{align*}
and we omit $r$ from the notation when the sum is over all integers $1\le n \le N$.
 
We can split $S(N, F)$ in two sums as follows
 \begin{align}
 \label{eq:firstsplit}
 S(N,F) = \sum_{r|q} S_r(N,F) = \sum_{\substack{ r|q \\ r <q^{1/d(d+1)}}} S_r(N,F) + \sum_{\substack{ r|q \\ r \ge q^{1/d(d+1)}}} S_r(N,F).
 \end{align}

 We can bound trivially the second sum of the right hand side of equation (\ref{eq:firstsplit}) using Lemma \ref{lemma:rhoprop}.  Indeed,
\begin{align}
 \label{eq:rbig}
\nonumber \left| \sum_{\substack{ r|q \\ r \ge q^{1/d(d+1)}}} S_r(N,F) \right|  \le  \sum_{\substack{ r|q \\ r \ge q^{1/d(d+1)}}} \sum_{\substack{n=1 \\ (n,q)=r}}^N \varrho(n)&  =  \sum_{\substack{ r|q \\ r \ge q^{1/d(d+1)}}} \sum_{\substack{n=1 \\ (n,q/r)=1}}^{N/r} \varrho(rn)\\
 & \ll_{\epsilon}  \frac{N}{q^{1/d(d+1) -\epsilon}}.
 \end{align}
 
We turn our attention to $S_r(N,F)$ when $r\le 1/d(d+1)$.  Let  $b_j = \frac{a_j\cdot q}{q_j},$ for $1\le j \le d$, $\tilde{F}(x) = \sum_{j=1}^d \frac{b_j}{q_j}x^j$ and $Q(x) = F(x) - \tilde{F}(x)$. Observe that integration by parts yields 
\begin{align}
\label{eq:mainterm}
\begin{split}
S_r(N, F) &= \ e\left(  Q(N) \right)  S_r(N, \tilde{F}) \\&-  2\pi i \sum_{j=1}^d j\beta_j  \int_{1}^N u^{j-1} e\left( Q(u) \right) S_r(u, \tilde{F})du,
\end{split}
\end{align}
where $\beta_j$ is defined by 
\begin{align}
\label{eq:betajdef}
\alpha_j=\frac{a_j}{q_j}+\beta_j.
\end{align}
In particular, from~\eqref{eq:aprox}
\begin{align}
\label{eq:betabound1}
|\beta_j|\le \frac{ (\log{N})^B}{q_jN^j}.
\end{align}
The above reduces to analysing $S_r(N, \tilde{F})$ instead of $S_r(N, F)$.

Observe that since we are summing over multiples of $r$, we can rewrite $e(\tilde{F}(n))$ as $e(\tilde{F_r}(n))$, where $\tilde{F_r}(n):=\tilde{F}(rn)$, so that the latter  is a periodic function modulo $q/r$. We let $q/r=r'$ and we decompose $S_r (u, \tilde{F_r})$ using Dirichlet characters modulo $r'$ as follows:

\begin{equation}
\begin{split}
S_r(N, \tilde{F})& = \sum_{x\bmod{r'}} e(\tilde{F}_r(x)) \sum_{ \substack{ m \le N/r \\ (m,r')=1 \\ m\equiv x\bmod{r'} }} \varrho(rm) \\
& = \frac1{\varphi(r')}\sum_{x\bmod{r'}} e(\tilde{F}_r(x))  \sum_{ \chi \bmod{r'}} \overline{\chi(x)} \sum_{m\le N/r} \chi(m) \varrho(rm).
 \end{split}
\end{equation}

We will prove: 

\begin{prop}
\label{prop:dirich}
Let $\chi$ be a Dirichlet character $\bmod{r'}$. There exist constants $c>0$, $0\le \delta(\chi) \le \lambda r $ such that
\begin{align*}
\sum_{m\le N/r} \chi(m)\varrho(rm) = \delta(\chi) \frac{N}{r} + O(Ne^{-\frac{1}{c} \sqrt{\log{N}}}).
\end{align*}
Moreover, if $\delta(\chi) \ne 0$ then $\chi$ has conductor  $h| \operatorname{disc}(p)$.
\end{prop}
We next show that Proposition~\ref{prop:dirich} implies we can conclude the proof of Proposition \ref{prop:main}. Indeed,  we obtain
\begin{align*}
S_r(N, \tilde{F}) = \frac{N}{r\varphi(r')}\sum_{x\bmod{r'}} e(\tilde{F}_r(x))  \sum_{ \chi \bmod{r'}} \overline{\chi(x)} \delta(\chi)  + O(N e^{-\frac{1}{c'}\sqrt{\log{N}}}),
\end{align*}
which substituted into~\eqref{eq:mainterm} and using~\eqref{eq:betabound1} implies 
\begin{align*}
S_r(N, F) &= \ e\left(  Q(N) \right)  S_r(N, \tilde{F}) \\&-  C\int_{1}^N u\left( 2\pi i \sum_{j=1}^d j\beta_j u^{j-1}\right) e\left( Q(u) \right)du+ O(N e^{-\frac{1}{c}\sqrt{\log{N}}}),
\end{align*}
with 
\begin{align*}
C= \frac{1}{r\varphi(r')}\sum_{x\bmod{r'}} e(\tilde{F}_r(x))  \sum_{ \chi \bmod{r'}} \overline{\chi(x)} \delta(\chi). 
\end{align*}
Note that integrating by parts
\begin{align*}
\int_{1}^N u\left( 2\pi i \sum_{j=1}^d j\beta_j u^{j-1}\right) e\left( Q(u) \right)du\ll N,
\end{align*}
and since  if $\delta(\chi) \ne 0$ then $\chi$ has conductor  $h| \operatorname{disc}(p)$, the above implies 
\begin{align}
\label{eq:SrFFNFNF}
S_r(N,F)\ll \frac{N}{r}\frac{1}{\phi(r')}\sum_{x\mod{r'}}\chi(x)e(\tilde F_r(x))+O(N e^{-\frac{1}{c}\sqrt{\log{N}}}),
\end{align}
for some $\chi$ with conductor  $h| \operatorname{disc}(p)$. 

Recall that 
\begin{align*}
\tilde F_r(x)=\tilde F(rx)=\sum_{j=1}^{d}\frac{b_jr^{j}}{q_j}x^{j}=\frac{\sum_{j=1}^{d}(qb_j r^{j}/q_j)x^{j}}{q}.
\end{align*}
Let $q'$ denote the smallest divisor of $q$ such that the polynomial 
$$P(x)=\sum_{j=1}^{d}(qb_j r^{j}/q_j)x^{j},$$ 
is constant mod $q/q'$. Note that for each $1\le j\le d$
$$\frac{q b_jr^{j}}{q_j}\equiv 0 \mod{q/q'}.$$
Since $(b_j,q_j)=1$, this implies 
$$q'> \frac{q_j}{r^{j}}\ge \left(\prod_{j=1}^{d}\frac{q_j}{r^{j}}\right)^{1/d}\ge \frac{q^{1/d}}{r^{(d+1)/2}}\ge q^{1/2d},$$
provided $r\le q^{1/d(d+1)}.$

It follows from work of Cochrane and Zheng~\cite{CZ} and the Chinese remainder theorem that 
\begin{align*}
\sum_{x\mod{r'}}\chi(x)e(\tilde F_r(x))\ll \frac{r'}{q'}(q')^{1-1/d}\ll \frac{r'}{q^{1/2d(d+1)}}.
\end{align*}

Combining the above with~\eqref{eq:SrFFNFNF} gives 
\begin{equation*}
S_r(N, F) \le \frac{N}{q^{1/2d(d+1)}} \quad \text{provided} \quad r\le q^{1/d(d+1)}.
\end{equation*}
Summing over $r < q^{1/d(d+1)}$ we get
\begin{align*}
\sum_{\substack{r|q \\ r < q^{1/d(d+1)}}} S_r(N,\tilde{F}) & \le \frac{N}{q^{1/2d(d+1)+o(1)}}.
\end{align*}

Together with~\eqref{eq:firstsplit} and~\eqref{eq:rbig} this concludes the proof of Proposition \ref{prop:main}. We will now focus on the proof of Proposition \ref{prop:dirich}. 
In order to understand the sum
\begin{equation}
\label{eq:sum}
\sum_{n \le N/r} \varrho(rn)\chi(n),
\end{equation} we study the Dirichlet series for $\chi \bmod{r'}$
\begin{equation*}
D_r(s, \chi ) = \sum_{n = 1}^\infty \varrho(rn)\chi(n) n^{-s}.
\end{equation*}

\subsection{Decomposing $D_r(s,\chi)$}
Obverse that $D_r(s,\chi)$ is absolutely convergent for $\operatorname{Re}(s) >1$.
Our next goal is to extend $D_r(s,\chi)$ to the left of $\operatorname{Re}(s) > 1$ so we can use contour integration to estimate the sum (\ref{eq:sum}).  We will do this by decomposing $D_r$ in terms of well-known Artin L-functions. With this in mind, we fix some notation.

 Denote by $K_f$ the splitting field of $f$ in $\C$, by $G$ the Galois group of $K_f$ over $\Q$ and by $\Q(e(1/r'))$ the cyclotomic field generated by $r'$-roots of unity. Observe that the latter is a Galois extension of $\Q$, with Galois group $C_{r'}$ isomorphic to $(\Z/r'\Z)^{\times}$.

We also consider the compositum of $K_f$ and $\Q(e(1/r'))$ and denote it by $K_{f,r'}$ and its Galois group over $\Q$ by $G_{r'}$.  Observe that there is a natural injection
\begin{align*}
G_{r'} \rightarrow G \times C_{r'} \simeq G \times (\Z/r'\Z)^\times.
\end{align*}
Moreover, observe that from the extension theorem for field automoprhism, it follows that the projections $p_1: G_{r'} \rightarrow G$ and $p_2: G_{r'} \rightarrow (\Z/r'\Z)^\times$ are surjective.

We denote by $\hat{G}$ the finite set of isomorphism classes of complex irreducible representations of $G$ and for $\pi \in \hat{G}$ we write $\chi_{\pi}$ for the character of $\pi$.

Note that $\chi$ can be viewed as a character of $(\Z/r'\Z)^\times$,  and consequently as a character $\eta: C_{r'} \rightarrow \C^\times$.  It is known that $\eta(\sigma_p) = \chi(p)$ for $p\nmid r'$ and $\sigma_p$ the Frobenius automorphism at $p$.

So, we consider the representations $p_1^\ast\pi = \pi \circ p_1$ and $p^\ast_2\eta = \eta \circ p_2$ of $G_{r'}$ and observe that their tensor product satisfies
\begin{align}
\label{eq:tensor}
\operatorname{tr}(p_1^\ast \pi \otimes p_2^\ast \eta(\sigma_p)) = \chi_\pi(\sigma_p)\chi(p),
\end{align}
for $p\nmid \operatorname{disc}(f)$ and $p\nmid r'$. 

\begin{prop}
\label{prop:rep}
$D_r(s,\chi)$ has an expression
\begin{align*}
D_r(s,\chi) = F_r(s,\chi) \prod_{\pi \in \hat{G}} E_{r,\pi}(s,\chi) ( L(s, K_{f,r'}/\Q, p_1^\ast \pi \otimes p_2^\ast \eta))^{m_\pi},
\end{align*}
for $m_\pi \ge 0$ an integer and, for any $\epsilon>0$, $F_r\ll_{\epsilon} q^{\epsilon}$
 and $E_{r, \pi}\ll q^{\epsilon}$ for $\operatorname{Re}(s) \ge 3/4$.
\end{prop}

To prove the proposition we will need the following lemma:
\begin{lemma}
\label{lemma:dec}
 Denote by $\pi_f$ the permutation representation of $G$ acting on the set of roots of $f$ in $\C$ and let
\begin{align*}
\pi_f = \bigoplus_{\pi \in \hat{G}} m_\pi \pi
\end{align*} 
 be its decomposition in irreducible representations, where $m_\pi \ge 0$ are integers. For all $p \nmid \operatorname{disc}(f)$ it holds
 \begin{align}
\label{eq2}
\varrho(p) =\sum_{\pi} m_\pi \cdot \chi_\pi(\sigma_p),
\end{align}
for $\sigma_p \in G$ the Frobenius automorphism at $p$.
\end{lemma}
\begin{proof}

The proof follows by noting that $\varrho(p)$ is the number of fixed points of the Frobenius automorphism $\sigma_p$ at $p$,  which is also the character at $\sigma_p$ of the permutation representation.

\end{proof}

\begin{proof}[Proof of Proposition \ref{prop:rep}]

First we observe that $D_r$ can be decomposed the following way
\begin{align*}
D_r(s, \chi) & = \sum_{e| r^{\infty}} \sum_{(k,r)=1} \chi(ek) \varrho(erk) (ek)^{-s} \\
& = \sum_{e|r^{\infty}} \varrho(er) \chi(e) e^{-s} \sum_{(k,r)=1} \chi(k) \varrho(k)k^{-s},
\end{align*}
where the notation $e|r^{\infty}$ means that $e$ runs over all possible products of powers of the primes that divide $r$.

We denote by
\begin{align*}
E_r(s,\chi)&=\sum_{e|r^{\infty}} \varrho(er) \chi(e) e^{-s} \\
\tilde{D}_r(s,\chi)&=\sum_{(k,r)=1} \chi(k) \varrho(k)k^{-s},
\end{align*}
and using Lemma \ref{lemma:rhoprop} and the property $\varrho(p^n) \le C,$ for all $n\ge 1$ and $p$ prime, we observe that $E_r(s,\chi) \ll r^\epsilon$  uniformly for $\operatorname{Re}(s) \ge 3/4$ and that $\tilde{D}_r(s,\chi)$ converges absolutely for $\operatorname{Re}(s)>1$.

To be able to explore the properties of $\varrho$, we factor the primes that divide the discriminant of $f$ in $\tilde{D}_r(s,\chi)$ and, since $r|q$, we also factor the remaining primes that divide $q$. This yields
\begin{align*}
\tilde{D}_r(s,\chi)=E_{1,r}(s,\chi) \prod_{ \substack{p\nmid\operatorname{disc}(f) \\ p\nmid q}} \sum_{k\ge 0} \varrho(p^k)\chi(p)^kp^{-ks},
\end{align*}
where $E_{1,r}(s,\chi)$ is entire and bounded for $\operatorname{Re}(s) \ge 3/4$ by $C(\epsilon)q^{\epsilon}$. 

We proceed in a similar manner and factor out the prime powers with exponent bigger than $2$,  obtaining
\begin{align}
\label{eq:reduction}
\tilde{D}_r(s,\chi)=E_{1,r}(s,\chi)E_{2}(s,\chi) \prod_{ \substack{p\nmid\operatorname{disc}(f) \\ p\nmid q}} (1+\varrho(p)\chi(p)p^{-s}),
\end{align}
for $E_2(s,\chi)$ holomorphic and bounded in vertical strips by $C(\epsilon)q^\epsilon$ for $\operatorname{Re}(s)\ge 3/4$. 

Now we can use Lemma \ref{lemma:dec} and a further decomposition to write
\begin{align*}
\tilde{D}_r(s,\chi) = E_{1,r}(s,\chi)E_{2}(s,\chi)E_{3,r}(s,\chi) \prod_{\pi \in \hat{G}} \prod_{p} (1+\chi_{\pi}(\sigma_p) \chi(p)p^{-s})^{m_\pi},
\end{align*}
for $E_{3,r}$ holomorphic and uniformly bounded in vertical strips by \newline $C(\epsilon, f)q^\epsilon$ for $\operatorname{Re}(s) \ge 3/4$.  Set $F_r = E_{1,r}E_{2}E_{3,r}$. To conclude, we use equality (\ref{eq:tensor}), and yet another similar decomposition as before and obtain 
\begin{align*}
\prod_p ( 1+ \chi_\pi(\sigma_p)\chi(p)p^{-s}) = E_{r,\pi}(s,\chi) L(s, K_{f,q}/\Q, p_1^\ast \pi \otimes p_2^\ast\eta).
\end{align*}
\end{proof}

Observe that now we are dealing with Artin L-functions, which are better understood than the original Dirichlet series. Since $p_1$ is surjective and $\pi$ is irreducible it follows that $p_1^\ast \pi $ is an irreducible representation. Furthermore,  since $p_2^\ast \eta$ is of dimension $1$ we can conclude that $p_1^\ast \pi \otimes p_2^\ast \eta$ is also an irreducible representation. Artin's conjecture states that this L-function is entire except if the representation is trivial - which can occur if $p_1^\ast \pi$ is the inverse of $p_2^\ast\eta$.

To avoid assuming Artin's conjecture we will use the Brauer induction theorem (see, e.g, \cite[Theorem 19]{Serre}) instead.  It states that for every subgroup $H$ of $G$ and 1-dimensional character $\beta_H:H \rightarrow \C^\times$,  there exist integers $n_{\pi,  \beta_H}$ such that
\begin{align*}
\chi_\pi = \sum_{H}\sum_{\beta_H} n_{\pi, \beta_H} \cdot {\operatorname{Ind}_H^G} \beta_H.
\end{align*} 

 Observe that this implies that $L(s, K_{f,q}/\Q, p_1^\ast \pi)$ can be represented as follows
 \begin{align*}
 L(s, K_{f,q}/\Q, p_1^\ast \pi) = L(s, K_f/\Q, \pi) & = \prod_{H}\prod_{\beta_H} L(s, K_f/\Q, \operatorname{Ind}_H^G \beta_H)^{n_{\pi, \beta_H}}\\
 & =  \prod_{H}\prod_{\beta_H}L(s, K_f/K_H, \beta_H)^{n_{\pi, \beta_H}},
 \end{align*}
where $K_H \subset K_f$ is the subfield fixed by $H$ - which implies that $H$ is the Galois gorup of $K_f$ over $K_H$.

We can now introduce the twist by the Dirichlet character. We denote by $H'=p_1^{-1}(H)$ and observe that we can write
\begin{align*}
p_1^\ast \pi = \bigoplus_{H } n_{\pi, \beta_H} \operatorname{Ind}_{H'}^{G_{r'}} (p_1^\ast \beta_H),
\end{align*}
and thus
\begin{align*}
p_1^\ast \pi \otimes p_2^\ast \eta = \bigoplus_{H}\bigoplus_{\beta_H} n_{\pi, \beta_H} \operatorname{Ind}_{H'}^{G_{r'}}(p_1^\ast \beta_H \otimes \operatorname{Res}_{H'}^{G_{r'}}p_2^\ast \eta).
\end{align*}
So denoting $\theta_{\beta_H}=p_1^\ast\beta_{H} \otimes \operatorname{Res}_{H'}^{G_{r'}}p_2^\ast \eta$,  we can finally write
\begin{align*}
L(s, K_{f,q}/\Q, p_1^\ast \chi \otimes p_2^\ast\eta) = \prod_{H} \prod_{\beta_H} L(s, K_{f,q} /K_H, \theta_{\beta_H})^{n_{\pi,  \beta_H}}, 
\end{align*}
thus $D_r$ has the following representation
\begin{align}
\label{repres}
D_r(s,\chi) = F_{r}(s,\chi) \prod_{\pi \in \hat{G}} E_{r, \pi}(s,\chi) \prod_{H}\prod_{\beta_H} L(s, K_{f,q} /K_H, \theta_{\beta_H})^{n_{\pi, \beta_H}}.
\end{align}
 Now $L(s, K_{f,q} /K_H, \theta_{\beta_H})^{n_{\pi, \beta_H}}$ are Artin L-functions of dimension one, which we know that are entire except for a pole in $s=1$ if $\theta_{\beta_H}$ is the trivial one dimensional character. 

\subsection{Bounds on the Artin L-functions near $s=1$}

We denote by $q_{\pi,i}$ the conductor of $L(s, K_{f,q} /K_H, \theta_{\beta_H})$ and observe that we can bound it as follows, (see, e.g. \cite{BH}), 
$$q_{\pi,\beta_H} \le \operatorname{cond}(p_1^\ast\beta_{H}) \cdot \operatorname{cond}(\operatorname{Res}_{H'}^{G_q}p_2^\ast \eta) \le M q,$$
where $M = \sup_{\pi,\beta_H} \operatorname{cond}(p_1^\ast\beta_{H})$.
It is well known that $L(s, K_{f,q} /K_H, \theta_{\beta_H})$ can have at most one real zero $\beta$ in the region
\begin{align}
\label{eq:zerofree}
\sigma \ge 1 - \frac{c}{\log( {q_{\pi, \beta_H}} (|t|+3))},
\end{align}
for  $s=\sigma + it \in \C$ and $c$ a universal constant, (see, e.g. \cite[Theorem 5.35]{I-K}). Note that this zero can only exist if $\theta_i$ is a quadratic character.

To simplify the notation, in what follows we let $g= \theta_{\beta_H}$  and we omit the field $K_{f,q} /K_H$. Denote by $\beta$ the possible exception in region (\ref{eq:zerofree}) and let $\boldsymbol{q}$ be the conductor of $g$. Also, we set $r_1=1$ if $g$ has an exceptional zero $\beta$ and $r_1 = 0$ otherwise.  Analogously, $r_2=1$ is $g$ is the trivial character, and consequently has a pole in $s=1$, and $r_2=0$ if not.

\begin{prop}
\label{prop1}
Let $s=\sigma + it \in \C$ satisfy the conditions:
\begin{align}
\label{condition1}
\sigma \ge 1 & - \frac{c}{10\log( M{q}(|t|+3))} \\
\label{condition2}
| s - \beta |& \ge \frac{1}{ 20 \log{(3M{q})}} \\
\label{condition3}
|s-1| & \ge \frac{1}{20 \log(3M{q})}.
\end{align}
It holds that 
\begin{equation*}
\begin{split}
L(s,g) & \ll \log( Mq(|t|+3))\\
\frac{1}{L(s,g)} & \ll \log(M q(|t|+3)),\\
\end{split}
\end{equation*}
where the implicit constant does not depend on the parameters.

If $g$ is not trivial condition (\ref{condition3}) can be dropped. Condition (\ref{condition2}) can be dropped if $g$ has no exceptional zero.
\end{prop}

\begin{proof}

We set $\gamma(s,g) = \pi^{-s/2} \gamma\left(\frac{s+\kappa_g}{2}\right)$, where $\kappa_g=0$ or $1$, and let $\Lambda(s,g) = \boldsymbol{q}^{s/2}\gamma(s,g) L(s,g)$  be the extended L-function. It is a known fact that $\Lambda(s,g)$ is a meromorphic function of order at most 1. Thus,  we can proceed as in \cite[Theorem 11.4]{M-V2} to obtain the result of the theorem.

\end{proof}

\subsection{Finishing the proof}
\begin{proof}[Proof of Proposition \ref{prop:main}]
To bound the sum $\sum_{n \le N/r} \varrho(rn)\chi(n)$ we first smooth it 
\begin{align*}
\sum_{\substack{n=1 \\ (n,q)=r}}^{N/r}\varrho(n) \chi(n) = \sum_{\substack{n=1 \\ (n,q)=r}}^\infty \varrho(n)\chi(n) \phi(n) + O(\tilde{N}),
\end{align*}

where $\phi$ is defined as

$$\phi(x) = \min\left( x, 1, 1+ \frac{N/r-x}{\tilde{N}} \right)$$
for $0 \le x \le N/r + \tilde{N}$ and $\phi(x) = 0$ for $x> N/r + \tilde{N}$,  $\tilde{N}>0$ to be chosen later. 
Observe that Mellin inversion implies that
\begin{equation}
\begin{split}
\label{eq:mellin}
&\sum_{\substack{n=1 \\ (n,q)=r}}^\infty  \varrho(n)\chi(n) \phi(n)   = \frac{1}{2\pi i}\int_{(3)} D_r(s,\chi) \hat{\phi}(s)ds \\
& =  \frac1{2\pi i } \int_{(3)} F_{r}(s,\chi) \prod_{\pi \in \hat{G}} E_{r, \pi}(s,\chi) (L(s, K_{f,r'}/\Q, p_1^\ast \pi \otimes p_2^\ast \eta))^{m_\pi}\hat{\phi}(s) ds \\
& = \frac1{2\pi i } \int_{(3)} F_{r}(s,\chi) \prod_{\pi \in \hat{G}} E_{r, \pi}(s,\chi)  \prod_{H}\prod_{\beta_H} L(s, K_{f,q} /K'_H, \theta_{\beta_H})^{n_{\pi, \beta_H}} \hat{\phi}(s)ds.
\end{split}
\end{equation}

Our goal is to shift the contour of integration and use the information  on the L-functions $ L(s, K_{f,q} /K'_H, \theta_{\beta_H})$ to obtain good bounds for $S_r(s,\tilde{F}_r)$.  Observe that we can find $c/2 \le Q \le c/10$ such that all the conditions of Proposition \ref{prop1} are satisfied for $$s \in  \mathcal{Z} = \{ s=\sigma + it, \sigma = 1 - \frac{Q}{\log{(M{q}(|t|+3))}} \}.$$

Moving the contour of integration in equation (\ref{eq:mellin}) to $\mathcal{Z}$ we get
\begin{equation}
\label{eq:final}
\begin{split}
\sum_{\substack{n=1\\(n,q)=r}}^{\infty}\varrho(n)\chi(n)\phi(n) & = \frac{1}{2\pi i}\int_{\mathcal{Z}}D_r(s,\chi) \widehat{\phi}(s)ds +\alpha \operatorname{Res}(D_r(s,\chi), 1) \cdot \frac{N}{r} \\
& + \tilde{\alpha}\operatorname{Res}(D_r(s,\chi), \beta) \cdot N^\beta , \\
\end{split}
\end{equation}
where $\alpha=1$ if there is $\pi \in \hat{G}$, for which  $p_1^\ast \pi \otimes p_2^\ast \eta$ is the principal character in $K_{f,r'}/\Q$ with $m_\pi = 1 $ and $\alpha = 0$ otherwise.  Likewise,  $\tilde{\alpha}=1$ if $ L(s, K_{f,q} /K'_H, \theta_{\beta_H})$ has an exceptional zero $\beta$  and $n_{\pi, \beta_H} = -1$, and $\tilde{\alpha} = 0$ otherwise.  Note from Proposition \ref{prop:rep} that there is at most one trivial character $p_1^\ast \pi \otimes p_2^\ast \eta$ and one quadratic character $\theta_{\beta_H}$.  Moreover, if $p_1^\ast \pi \otimes p_2^\ast \eta$ is the trivial character, then we observe that $\chi$ must have a conductor with modulus $h$ that divides $\operatorname{disc}(f)$.

 We bound each term of the right hand side of equation (\ref{eq:final}) separately. 
To bound the first term, we note that the Mellin transform of $\phi$ satisfies 
\begin{align*}
\widehat{\phi}(s) = \int_{0}^{N+\tilde{N}} \phi(z) z^{s-1}dz \ll \frac{N^{\sigma}}{|s|} \min\left( 1 , \frac{N}{|s|\tilde{N}}\right),
\end{align*}
(see \cite[Theorem 5.12]{I-K}). Thus, 
\begin{align*}
\int_{\mathcal{Z}} D_r(s,\chi) \hat{\phi}(s)ds \ll \int_{\mathcal{Z}} | D_r(s,\chi)|\frac{N^{\sigma}}{|s|} \min\left( 1 , \frac{N}{|s|\tilde{N}}\right) |ds|.
\end{align*}
We let $T= N/r\tilde{N}$ and $\sigma(T) = 1 - Q/\log(Mq(T+3))$ and note that the observation above,  Proposition \ref{prop:rep} and 
Proposition \ref{prop1} imply that
\begin{align*}
\int_{\mathcal{Z}} D_r(s,\chi) \hat{\phi}(s)ds \ll N^{\sigma(T)}q^{Y\epsilon},
\end{align*}
for a constant $Y>0$ that only depends on $f$.

We now deal with the last term $  \tilde{\alpha} \operatorname{Res}(D_r(s,\chi), \beta) \cdot N^{\beta}$.  Siegel proved (see, e.g. \cite[Theorem 5.28]{I-K})  that if the exceptional zero $\beta$ exists then for every $\epsilon >0$ there exists a constant $C(\epsilon)>0$ such that 
\begin{align}
\label{eq:zerobound}
\beta \le 1 - \frac{c(\epsilon)}{\operatorname{cond}(\theta_{i})^{\epsilon} }\le 1 - \frac{c(\epsilon)}{Mq^{\epsilon}}.
\end{align}
To bound the residue at $\beta$ we notice that if $\theta_{\pi, \beta_H}$ is not trivial we can use Proposition \ref{prop1} and obtain
\begin{align*} 
| L(s, K_{f,q} /K'_H, \theta_{\beta_H})|^{n_{\pi, i}} \ll |\log{3Mq}|^{|n_{\pi, \beta_H}|}.
\end{align*}
If $\theta_{\beta_H}$ is trivial then 
\begin{align*}
| L(s, K_{f,q} /K'_H, \theta_{\beta_H})|^{n_{\pi, \beta_H}} \ll_{\epsilon} q^{\epsilon|n_{\pi, \beta_H}|},
\end{align*}
and for the residue given by the L-function with quadratic character we have
\begin{align*}
\operatorname{Res}\left( \frac1{ L(s, K_{f,q} /K'_H, \theta_{\beta_H})}, \beta \right) = \frac1{ L(s, K_{f,q} /K'_H, \theta_{\beta_H})}.
\end{align*}
We can deduce that $ L'(s, K_{f,q} /K'_H, \theta_{\beta_H}) \gg 1$ from the last part of Theorem 11.4 in \cite{M-V2}. Putting everything together we obtain
\begin{align*}
\operatorname{Res}(D_r(s,\chi), \beta) \cdot N^{\beta} \le q^{\epsilon} N^{1 - C(\epsilon)/q^{\epsilon}}.
\end{align*}

Pick $T=\exp\left(\frac{1}{3} \sqrt{\log{N}}\right)$. Since $q \le (\log{N})^B$, the considerations above imply that
\begin{align}
\label{ref:final}
S_r(N, \tilde{F}_r) = \alpha \operatorname{Res}(D_r(s,\chi), 1)\cdot \frac{N}{r} + O(N e^{-\frac{1}{c(\epsilon)} \sqrt{\log{N}}}),
\end{align}
for a constant $c(\epsilon)$ and $\epsilon>0$ sufficiently small. Since
\begin{align*}
\sum_{n\le N}\varrho(n) \le \lambda N, 
\end{align*}
and the error term in equation (\ref{ref:final}) is $o(N)$  we conclude that $$\operatorname{Res}(D_r(s,\chi), 1) \le \lambda r$$. 
\end{proof}

\section{Sharpness of Theorem \ref{theo:main}}
\label{sec:sharp} 
Following the approach of \cite{M-V}, we will construct a completely multiplicative function $f=f_{N,F}$, with $|f(n)| \le 1$, such that
\begin{align}
\label{eq:sharp}
\left|\sum_{n\le N} f(n)e\left(F(n)\right)\right| \ge \frac{1}{10} \frac{N}{\log{N}}.
\end{align} 
We first observe that the function  $G: \C \rightarrow \C$ given by
\begin{align*}
G(z) = \sum_{n\le N} z^{\Omega(n)}e(F(n)) + \sum_{ \frac{N}{2} < p \le N} (1-ze(F(p)))
\end{align*} 
is entire and its value at zero satifies
\begin{align*}
G(0) = \sum_{ \frac{N}{2} < p \le N} 1 \ge \frac1{10} \frac{N}{\log{N}},
\end{align*} 
for $N$ sufficiently large.
Thus, by the maximum modulus principle, there exists a $z_0 \in \C$, $|z_0| = 1$, such that $|G(z_0)| \ge |G(0)|$. 

Define the completely multiplicative function $f$ by  $f(p) = z_0$, for $ p \le N/2$, and $f(p)=e(-F(p))$, for $p > N/2$. To conclude that equation \eqref{eq:sharp} is satisfied, we just observe that
\begin{align*}
\sum_{n\le N} f(n)e\left(F(n)\right) & = \sum_{n\le N} z_0^{\Omega(n)}e(F(n)) + \sum_{ \frac{N}{2} < p \le N} (1-z_0e(F(p)))\\
& = G(z_0) \gg \frac{N}{\log{N}}.
\end{align*}

\end{document}